\documentclass[twoside,11pt,intlimits]{amsart}

\usepackage{amsmath,amssymb,amsthm}
\usepackage{amsfonts}
\usepackage{amsxtra}
\usepackage[usenames,dvipsnames]{xcolor}
\usepackage{stmaryrd}
\usepackage{wasysym}
\usepackage{xspace}
\usepackage{graphicx}

\usepackage{todonotes}

\usepackage{enumitem}
\setenumerate{label={\rm (\alph{*})}}

\usepackage{ifthen}

\newtheorem{theorem}{Theorem}[section]
\newtheorem{algorithm}{Algorithm}[section]
\newtheorem{lemma}[theorem]{Lemma}
\newtheorem{proposition}[theorem]{Proposition}
\newtheorem{corollary}[theorem]{Corollary} 

\newtheorem{definition}[theorem]{Definition}

\newtheorem{remark}[theorem]{Remark}

\definecolor{lightgrey}{rgb}{.7,.7,.7}

\DeclareMathOperator*{\argmin}{arg\,min}

\numberwithin{equation}{section}

\newcommand{\energy}{\mathcal{G}}
\newcommand{\transop}{\mathcal{J}}
\newcommand{\ipol}[1][\tria]{\mathcal{I}_{#1}}

\newcommand{\refine}{\ensuremath{{\textsf{REFINE}}}}
\newcommand{\refd}[2]{\ensuremath{{\textsf{ref'd}(#1; #2)}}}
\newcommand{\midpoint}[1]{\ensuremath{{\textsf{mid}(#1)}}}

\newcommand{\Curl}{\mathrm{Curl}}

\providecommand{\osc}{\operatorname{osc}}

\newcommand{\dx}{\ensuremath{\,\mathrm{d}x}}
\newcommand{\ds}{\,\mathrm{d}s\xspace}

\providecommand{\tria}{\ensuremath{\mathcal{T}}}
\providecommand{\bbT}{\mathbb{T}}
\providecommand{\bbThat}{\widehat{\bbT}}
\providecommand{\sides}{\mathcal{S}}
\providecommand{\nodes}{\mathcal{N}}
\providecommand{\tildeMk}[1][k]{\widetilde{\mathcal{M}}_{#1}}
\providecommand{\Mk}[1][k]{\mathcal{M}_{#1}}

\providecommand{\VoT}[1][\tria]{\ensuremath{\mathbb{V}_{#1}}}

\providecommand{\vT}[1][\tria]{\ensuremath{v_{#1}}}
\providecommand{\uT}[1][\tria]{\ensuremath{u_{#1}}}
\providecommand{\bvT}[1][\tria]{\ensuremath{\boldsymbol{v}_{#1}}}
\providecommand{\bwT}[1][\tria]{\ensuremath{\boldsymbol{w}_{#1}}}
\providecommand{\buT}[1][\tria]{\ensuremath{\boldsymbol{u}_{#1}}}
\providecommand{\bfu}{\ensuremath{\boldsymbol{u}}}
\providecommand{\bfv}{\ensuremath{\boldsymbol{v}}}
\providecommand{\bff}{\ensuremath{\boldsymbol{f}}}

\providecommand{\estT}[1][\tria]{\ensuremath{\mathcal{E}_{#1}}}

\providecommand{\estTbar}[1][\tria]{\ensuremath{\bar{\mathcal{E}}_{#1}}}

\providecommand{\Cleq}{\ensuremath{\lesssim}}
\providecommand{\Cgeq}{\ensuremath{\gtrsim}}

\providecommand{\id}{\operatorname{id}}

\providecommand{\tangente}{{\boldsymbol{t}}}

\providecommand{\settmp}[2]{{#1\{{#2}#1\}}}
\providecommand{\set}[1]{\settmp{}{#1}}

\providecommand{\abstmp}[2]{{#1\lvert{#2}#1\rvert}}
\providecommand{\abs}[1]{\abstmp{}{#1}}

\providecommand{\normtmp}[2]{{#1\lVert{#2}#1\rVert}}
\providecommand{\norm}[2][\Omega]{\normtmp{}{#2}_{#1}}

\providecommand{\setN}{\ensuremath{\mathbb{N}}}
\providecommand{\setR}{\ensuremath{\mathbb{R}}}

\providecommand{\jumptmp}[2]{#1\llbracket{#2}#1\rrbracket}
\providecommand{\jump}[1]{\jumptmp{}{#1}}

\newcommand{\hT}[1][\tria]{h_{#1}}

\newcommand{\CR}[1][\tria]{\mathrm{CR}^1_0( #1)}

\newcommand{\QT}[1][\tria]{\mathrm{Q}_0(#1)}
\newcommand{\gradnc}{\nabla_{\negthinspace\mathtt{NC}}}
\newcommand{\divnc}{\operatorname{div}_{\mathtt{NC}}}
\newcommand{\av}{J}
\newcommand{\compop}[1][\tria]{\mathcal{K}_{#1}}
\newcommand{\Vcr}[1][\tria]{V_{\mathrm{CR}}(#1)}

\begin{document}

\author[C.~Kreuzer]{Christian Kreuzer}
\address{Christian Kreuzer,
 Fakult{\"a}t f{\"u}r Mathematik,
 Ruhr-Universit{\"a}t Bochum,
 Universit{\"a}tsstrasse 150, D-44801 Bochum, Germany
 }
\urladdr{http://www.ruhr-uni-bochum.de/ffm/Lehrstuehle/Kreuzer/index.html}
\email{christian.kreuzer@rub.de}

\author[M.~Schedensack]{Mira Schedensack}
\address{Mira Schedensack,
 Institut f\"ur Mathematik, 
            Humboldt-Universit\"at zu Berlin, 
            Unter den Linden 6, D-10099 Berlin, Germany 
 }
\email{schedens@math.hu-berlin.de}

\title[Instance optimal nonconforming AFEMs]{Instance optimal
  Crouzeix-Raviart\\ adaptive finite element methods for the\\ Poisson
  and Stokes problems}

\begin{abstract}
  We extend the ideas of Diening, Kreuzer, and Stevenson [Instance optimality of the
  adaptive maximum strategy, Found. Comput. Math. (2015)],   
  from conforming
  approximations of the Poisson problem 
  to nonconforming Crouzeix-Raviart approximations of the Poisson and the
  Stokes problem in 2D. As a consequence, we obtain instance optimality of
  an AFEM with a
  modified maximum marking strategy.
\end{abstract}

\keywords{Adaptive finite element methods, optimal cardinality,
  convergence, maximum marking strategy, nonconforming finite element methods}

\subjclass[2010]{65N30, 65N12, 65N50, 65N15, 41A25}

\maketitle

\section{Introduction}
\label{sec:introduction}

In recent years there has been an immense progress in the convergence and optimality
analysis of adaptive finite element methods (AFEMs) for the numerical
solution of partial differential equations. Such methods typically
consist of a loop
\begin{align}\label{SEMR}
  \textsf{SOLVE}\rightarrow   \textsf{ESTIMATE} \rightarrow
  \textsf{MARK} \rightarrow \textsf{REFINE}.
\end{align}
In \textsf{ESTIMATE} local indicators are computed. These indicators
are used in \textsf{MARK} in order to mark elements for refinement
based on some marking strategy. 

In 1996, D\"orfler  \cite{Doerfler:96} introduced a bulk chasing
marking strategy, which guarantees linear convergence for sufficiently
fine initial triangulations. The latter condition was removed by
Morin, Nochetto \& Siebert in \cite{MoNoSi:00,MoNoSi:02}.
Binev, Dahmen, and DeVore \cite{BiDaDeV:04} supplemented this AFEM by a
  so-called \emph{coarsening} routine and showed instance optimality
  of the resulting method; compare also with  \cite{Binev:07}. 
This means that, for a triangulation generated by the
  AFEM the corresponding error is, up to some multiplicative constant, smaller than 
the error corresponding to any (admissible) triangulation with,  up to some second fixed
multiplicative constant, at most the same number of 
newly generated triangles.
  We emphasize that this property is satisfied for every 
triangulation generated by the AFEM and is in particular non-asymptotic.

Relying on the
same strategy, 
Stevenson~\cite{Stevenson:07} and later Cascon, Kreuzer, Nochetto \& Siebert 
\cite{CaKrNoSi:08} proved optimal convergence rates of an AFEM,
without coarsening
when the bulk parameter is below some problem
dependent threshold. The latter result refers to convergence rates for
the so called total error, which is the sum of the error and the
oscillation. Together with the linear convergence this even entails 
optimal
computational complexity of the method; compare with \cite{Stevenson:07}.
These features promoted the popularity of
D\"orfler's marking strategy in the numerical analysis of conforming
and nonconforming AFEMs for
elliptic problems;
compare e.g.\ with \cite{MoNoSi:03,BelenkiDieningKreuzer:12,KreuzerSiebert:11,
ChHoXu:06,MekchayNochetto:05,FeiFuehPraet:2013,BonitoNochetto:2010,
BeckerMaoShi:10,Rabus:2010,CarstensenHoppe:06b} 
as well as with the overview articles
\cite{NoSiVe:09,CarstensenFeischlPagePraetorius:2014} and the references
therein. However, it can be observed in numerical experiments 
that D\"orfler's marking strategy may indeed fail to be
optimal for large marking parameters; compare e.g. with
  \cite[\S8]{BraesFraunholzHoppe:2004}. Furthermore,
the sharp value of the threshold is in general unknown respectively 
the
theoretical bounds are very conservative. 
This makes the choice of the `right' marking parameter delicate. 

In contrast to this, only very little is known about the \emph{maximum
  marking strategy}, 
which marks all elements whose
indicators are bigger than the marking parameter times the maximal
indicator; compare e.g.\ with \cite{MoSiVe:08,Siebert:11}.
This strategy is widely-used in engineering and scientific
computation and numerical computations indicate that
its performance is 
robust with respect to the choice of the marking parameter. 
This observation was recently verified theoretically by Diening, Kreuzer \&
Stevenson in \cite{DieKreuStev:15} who proved
instance optimality of a conforming AFEM
with modified maximum marking strategy for the Poisson problem in two
dimensions.
In particular, the maximum
strategy is applied to some modified indicators, which accumulate
effects according to conforming mesh refinements and the instance
optimality refers to the total error. We recall that,
apart from the oscillatory term in the total error, instance
optimality is a non-asymptotic property and thus implies convergence with optimal
rates.
The result relies on new ideas
exploiting sharp discrete efficiency and discrete reliability 
bounds, a newly
developed tree structure of newest vertex bisection as well as a
newly developed so-called lower diamond estimate for the Dirichlet energy.

The aim of this paper is to extend the ideas of \cite{DieKreuStev:15} to
Crouzeix-Raviart AFEMs for the Poisson and the Stokes equations. 
For the Stokes problem, standard conforming finite elements 
are based on the saddle point formulation and satisfy the  
incompressibility condition in a weak sense only; see \cite{BrezziFortin:91}.
In contrast to this, Crouzeix-Raviart finite elements 
allow for piecewise exactly divergence free velocity
approximations. 
The discretisation therefore becomes a minimising problem on the subspace
of nonconforming divergence free finite element functions. 
This fact was used in the proofs of optimal convergence rates of
nonconforming AFEMs with D\"orfler's marking strategy in
\cite{BeckerMao:11,CarstensenPeterseimRabus:2013,HuXu:2013}. 
In our case, the above observation allows to define generalised energies
for the Stokes problem similarly as for the Poisson problem. 
For both problems, the generalised energy 
is monotone under refinement, satisfies the lower diamond estimate, 
and the energy difference is equivalent to the sum of the nonconforming
total error, i.e., the sum of the nonconforming error and the
oscillation.  Based on an improvement of the so called
transfer operator from \cite{CarstensenGallistlSchedensack:2013}, we 
conclude the precise discrete efficiency and
discrete reliability bounds. According to \cite{DieKreuStev:15}, 
we conclude that an AFEM with modified marking strategy is
instance optimal for the nonconforming total error.

The outline of the paper is as follows. In 
Section~\ref{sec:abstract-setting} we present the abstract setting 
used in \cite{DieKreuStev:15} including the error bounds,
the lower diamond estimate, and the AFEM. 
In Section~\ref{sec:poisson} we shall verify this setting for the
Poisson problem and in Section~\ref{sec:stokes} we extend the ideas 
to the Stokes equations.
The appendix presents the required modifications of the 
transfer operator from~\cite{CarstensenGallistlSchedensack:2013}.

\section{Abstract Setting}
\label{sec:abstract-setting}
In this section we shall introduce the basic framework, 
which is used in \cite{DieKreuStev:15} in order to prove 
instance optimality of 
a conforming adaptive finite element method for
Poisson's problem.

\subsection{Refinement Framework}
\label{sec:refinement}
We consider mesh adaptation 
by newest vertex bisection (NVB) in two dimensions; compare with
\cite{Baensch:91,Kossaczky:94,Maubach:95,Traxler:97,BiDaDeV:04,Stevenson:08} as
well as \cite{NoSiVe:09,SchmidtSiebert:05} and the references therein.
We denote by $\tria_\bot$ a conforming initial or ``bottom''
triangulation of a polygonal domain $\Omega \subset \mathbb{R}^2$
such that the labelling of the newest vertices satisfies the {\em matching
  assumption}: 
\begin{gather}
  \label{ass:matching}
  \text{\parbox{10cm}{\em If, for $T,T' \in \tria_\bot$, $T \cap T'$
      is
      the refinement edge of $T$, \\[1mm]
      then it is the refinement edge of $T'$.}}
\end{gather} 
It is shown in \cite{BiDaDeV:04}, that 
such a labelling can be found for any conforming $\tria_\bot$.
We denote by $\bbT$ the set of all conforming refinements  
of $\tria_\bot$ that can be created by 
finitely many newest vertex bisections. Then $\bbT$ is uniformly shape
regular, i.e., there exists a constant $C=C(\tria_\bot)$, such that 
\begin{align*}
  \sup_{\tria\in\bbT}\max_{T\in\tria}
  \frac{\operatorname{diam}(T)}{\abs{T}^{1/2}}\le C<\infty. 
\end{align*}
For $\tria\in\bbT$, let $\sides(\tria)$ be the set of its sides and
$\nodes(\tria)$ the set of its nodes. We denote by  
$\hT:\Omega\to\mathbb{R}$, $\hT|_T:=h_T:=|T|^{1/2}$, 
$T\in\tria$, the piecewise
constant mesh-size function and by $h_S:=|S|$ the length of the side
$S\in\sides(\tria)$.

For $\tria,\tria_\star \in \bbT$, we write $\tria \leq \tria_\star$ or $\tria_\star
\geq \tria$, when  $\tria_\star$ is a {\em refinement} of $\tria$ or,
equivalently, when $\tria$ is a {\em coarsening} of $\tria_\star$. 
This defines a partial ordering on~$\bbT$. We denote by $\tria\wedge
\tria_\star\in\bbT$ the finest common coarsening and by
$\tria\vee\tria_\star\in\bbT$ the coarsest common refinement of $\tria$ and
$\tria_\star$. Then $(\bbT, \le)$ is a lattice with
bottom~$\tria_\bot$. Moreover, we can formally define a largest
refinement $\tria^\top:=\bigvee_{\tria\in\bbT}\tria$, and set
$\tria^\top\ge\tria$ for all $\tria\in\bbT$. Then $\widehat\bbT:=\bbT\cup
\{\tria^\top\}$ is a bounded lattice; compare with
\cite{DieKreuStev:15}. 
For the sake of a uniform presentation, we set
$\tria\setminus\tria^\top:=\tria$ and
$\sides(\tria)\setminus\sides(\tria^\top):=\sides(\tria)$.

Let  $\omega_S^\tria :=
\operatorname{interior} \bigcup\{T\in\tria\mid S\subset \partial T\}$ denote
the patch of $S\in\sides(\tria)$. 
For $\mathcal{U}\subset\tria\in\bbT$, we
define as
$\Omega(\mathcal{U}):=\operatorname{interior}(\bigcup_{T\in\mathcal{U}}T)$
the domain of $\mathcal{U}$. 
Thanks to basic properties of bisection, for $\tria,\tria_\star\in\bbT$,
$\tria_\star\ge\tria$, we have the mesh-size reduction property
\begin{align}\label{eq:meshsize-red}
  \hT[\tria_\star]^2|_{\Omega(\tria\setminus\tria_\star)}\le\frac12\,
  \hT^2|_{\Omega(\tria\setminus\tria_\star)}.
\end{align}

The following definition introduces diamonds of triangulations. 
\begin{definition} \label{d:LD} For
    $\set{\tria_1, \dots, \tria_m} \subset \bbT$, $m\in\setN$, we call
    $(\tria^\wedge,\tria_\vee;\tria_1, \dots, \tria_m)$ a {\em lower
      diamond} in $\bbT$ of size $m$, 
       when $\tria^\wedge=\bigwedge_{j=1}^m \tria_j$,
    $\tria_\vee=\bigvee_{j=1}^m \tria_j$, and the areas of {\em
      coarsening} $\Omega(\tria_j \setminus \tria_\vee)$ are pairwise
    disjoint.
  
\end{definition}

We suppose that a function $\refine$ is at our disposal such that 
for $\tria \in \bbT$ and $\mathcal{M} \subset \sides(\tria)$ we have
that $\tria_\star=\refine(\tria;\mathcal{M})$ is the {\em coarsest
  refinement of $\tria$ in $\bbT$ with $\mathcal{M} \cap
  \sides(\tria_\star)=\emptyset$}. To be more precise, we have 
$$
\refine({\tria};{\mathcal{M}}) \:= \bigwedge \set{\tria' \in \bbT\,:\,
  \tria' \geq\tria,\, \mathcal{M} \cap \sides(\tria')=\emptyset}.
$$
Note that this routine can be implemented using
iterative or recursive bisection; see
\cite{Baensch:91,Kossaczky:94,Maubach:95,Mitchell:89,Traxler:97,Stevenson:08}.

Let $\{\tria_k\}_{k\in\setN_0}\subset\bbT$, with
$\tria_0=\tria_\bot$ and $\mathcal{M}_k\subset\sides(\tria_k)$, $k\in\setN_0$, such
that $\tria_{k+1}=\refine({\tria_k};{\mathcal{M}_k})$,
$k\in\setN_0$. Then 
\begin{align}\label{eq:BDD}
  \#(\tria_k\setminus\tria_\bot)\leq C \sum_{j=0}^{k-1}\#\mathcal{M}_j,
\end{align}
with a constant $C>0$ depending solely on $\tria_\bot$; compare with
\cite{BiDaDeV:04,DieKreuStev:15}.

\begin{remark}\label{rem:Pop}
  We note that, in order to obtain~\eqref{eq:BDD}, 
  for an iterative realization of the NVB the matching
  assumption~\eqref{ass:matching}  is not necessary in 2D; see
  \cite{KarkulikPavlicekPraetorius:13}.   
  However, assumption~\eqref{ass:matching} is needed to properly
  define a so called \emph{population} tree structure on
  triangulations, which is a newly developed crucial tool in the instance
  optimality analysis of~\cite{DieKreuStev:15}.
\end{remark}

\subsection{Problem Setting}
\label{sec:energy}
In this section we present the problem properties, which have been
used in the instance optimality analysis of \cite{DieKreuStev:15}.

\paragraph{\textbf{Energy}} We assume that there exists
an energy $\energy:\widehat\bbT\to\setR$,
which is monotonically decreasing with respect to the partial ordering of
$\widehat\bbT$, i.e., we have
\begin{align}\label{eq:emon}
  \energy(\tria_\star)\le\energy(\tria)\qquad\text{for
    all}~\tria,\tria_\star\in\bbThat, \tria\le\tria_\star. 
\end{align}
The overall aim is then to efficiently construct a
triangulation $\tria\in\bbT$, such that 
\begin{align}\label{eq:emin}
  \energy(\tria)-\energy(\tria^\top)\qquad\text{is {\em small}.}
\end{align}

\paragraph{\textbf{Error Bounds}} For $\tria\in\bbT$, the energy
difference is equivalent to the squared sum of some computable
local a posteriori indicators $\estT^2(S)\ge 0$, $S\in\sides(\tria$:
\begin{subequations}\label{eq:errbnd}
  \begin{align}\label{eq:errbnda}
    \energy(\tria)-\energy(\tria^\top)&\eqsim \estT^2(\sides(\tria)):=\sum_{S\in\sides(\tria)}\estT^2(S).
  \intertext{Moreover, if $\tria,\tria_\star\in\bbT$, 
  $\tria_\star\ge\tria$, then we have the \emph{localised} equivalence relation}
\label{eq:errbndb}
    \energy(\tria)-\energy(\tria_\star)&\eqsim
    \estT^2(\sides(\tria)\setminus\sides(\tria_\star)). 
  \end{align}
\end{subequations}

Here and in the following, the formula $A\Cleq B$ or $B\Cgeq A$ abbreviates 
that there exists some positive generic constant $C>0$, depending only
on properties of $\mathcal{G}$ and parameters of $\tria_\bot$, such
that $A\leq C B$; $A\eqsim B$ abbreviates $A\Cleq B\Cleq A$.

  \paragraph{\textbf{Lower Diamond Estimates}}
  We further assume, that the energy~$\energy: \bbThat \to \setR$ satisfies the {\em lower
      diamond estimate}, i.e., we have for any lower
    diamond 
    $(\tria^\wedge,\tria_\vee;\tria_1, \dots, \tria_m)$ in $\bbT$, $m\in\setN$, that 
  \begin{align}\label{eq:LDE}
    \energy(\tria^\wedge)-
    \energy(\tria_\vee) &\eqsim \sum_{j=1}^m \big(
    \energy(\tria_j)-
    \energy(\tria_\vee) \big).
  \end{align}
 
  \begin{remark}[Energy]\label{r:energy}
    The definition of $\energy(\tria)$ typically involves the finite
    element solution on $\tria$ and 
    is motivated by some problem specific
    energy. 
    Consider, e.g., the Poisson problem: The solution 
    $u\in H_0^1(\Omega)$ of 
    \begin{align*}
      -\Delta u=
      f~\text{in}~\Omega\qquad\text{and}\qquad u=0~\text{on}~\partial\Omega,
    \end{align*}
    minimises the
    Dirichlet energy
    \begin{align*}
      \mathcal{J}(v):=\int_\Omega\frac12|\nabla v|^2-fv\dx, \quad v\in
      H_0^1(\Omega)
    \end{align*}
over $H^1_0(\Omega)$.
For the standard $P_1$ conforming discretization in 
\cite{DieKreuStev:15}, we have that the finite element spaces are
nested under refinement, i.e., if $\tria_\star\ge\tria$, then
$\mathcal{J}(\tria)\ge \mathcal{J}(\tria_\star)$. Here we used the
abbreviation $\mathcal{J}(\tria):=\mathcal{J}(\uT)$ with $\uT$ being
the Galerkin approximation corresponding to $\tria$.
In order to ensure~\eqref{eq:errbnd} for the standard residual error
estimator, the energy $\energy(\tria)$ is then defined by 
$\mathcal{J}(\tria)$ plus some data dependent term (the element
residual in the scaled $L^2$-norm); compare with 
\cite[\S4]{DieKreuStev:15}.
According to 
\begin{align*}
  \frac12 \int_\Omega|\nabla u -\nabla v|^2\dx =
  \mathcal{J}(v)-\mathcal{J}(u)\qquad\text{for all}~v\in H_0^1(\Omega),
\end{align*}
we have that energy differences correspond to errors and the error 
  bounds in \eqref{eq:errbnda} are known as efficiency and reliability of the 
  error estimator. 

For the definition of the energy $\mathcal{G}$ in the case of
nonconforming Crouzeix-Raviart finite elements 
see Sections~\ref{sec:emin} and~\ref{sec:eminStokes}.
\end{remark}

\begin{remark}\label{rem:est}
  The error indicators are organised by edges for
  the following reasons:
  
  \begin{itemize}[leftmargin=0.5cm]
  \item For element based a posteriori indicators,
    the localised upper bound for the error of two nested discrete
    finite element solutions typically requires slightly
    more elements than the corresponding lower bound. Thus 
    an equivalence property as in~\eqref{eq:errbndb} cannot be strictly
    satisfied; compare e.g. with the localised upper bound in
    \cite{CaKrNoSi:08}, \cite{DieKreuStev:15}, and
    Section~\ref{sec:error-bounds} below.  
   \item Each midpoint of an edge $S\in\sides(\tria)$ can be
    identified with the \emph{new} nodes
    $\nodes(\tria^{++})\setminus\nodes(\tria)$ of the triangulation
    $\tria^{++}$ where all elements of $\tria$ are refined exactly
    twice (or equivalently all sides of $\tria$ are refined exactly
    once).  
    We recall from Remark~\ref{rem:Pop} that the nodes of a
    triangulation can be identified with the corresponding population,
    i.e., the side based indicators $\estT(S)$, $S\in\sides(\tria)$,
    can be interpreted as population based indicators by identifying
    $S$ with
    $\midpoint{S}\in\nodes(\tria^{++})\setminus\nodes(\tria)$.
    This fact is crucial in the instance optimality analysis of~\cite{DieKreuStev:15}.
  \end{itemize}

  \end{remark}

\subsection{An Instance Optimal AFEM}
\label{sec:afem}
In order to formally describe the adaptive method, 
for $\tria \in \bbT$ and $S \in \sides(\tria)$,  let
$$
\refd{\tria}{S}:=\sides(\tria) \setminus \sides(\refine({\tria};{S})),
$$
i.e., $\refd{\tria}{S} $ is
the subset of sides in $\sides(\tria)$ that are bisected in
the smallest refinement of $\tria$ in which
$S$ is bisected.

\begin{algorithm}[AFEM]\label{algo:AFEM}
  Fix $\mu \in (0,1]$ and set 
  $\tria_0:=\tria_\bot$ and $k=0$. The adaptive loop is an iteration
  of the following steps:
 \par{  
  \linespread{1.5}
    \sffamily
   \renewcommand{\arraystretch}{2}
    \begin{tabbing}
      (1) {ESTIMATE:}\quad\= \kill
      (1) {SOLVE}: \>solve the discrete problem  on $\tria_k$;\\ 
      (2) {ESTIMATE}: \> compute $\estTbar[\tria_k]^2
      :=\max\set{\estT[\tria_k]^2(\refd{\tria_k}{S}): S \in
        \sides(\tria_k)}$;
      \\
      (3) {MARK}: \> set $\Mk:=\emptyset$, ${\mathcal
        C}_k:=\sides(\tria_k)$, and $\tildeMk:=\emptyset$;
      \\
      \>while \= ${\mathcal C}_k \neq \emptyset$ do 
      \\
      \>\>select $S \in {\mathcal C}_k$;
      \\
      \> \>if $\estT[\tria_k]^2(\refd{\tria_k}{S} \setminus \tildeMk)
      \geq \mu \estTbar[\tria_k]^2$;
      \\
      \>\> then \=$\Mk:=\Mk \cup \{S\}$;\\
      \>\>\>$\tildeMk:=\tildeMk \cup \refd{\tria_k}{S}$;\\
      \>\> end if;\\
      \>\>${\mathcal C}_k:={\mathcal C}_k \setminus \refd{\tria_k}{S}$;\\
      \>end while;\\
      (4) {REFINE}: \>compute $\tria_{k+1}=\refine({\tria_k};{\Mk})$ and 
      increment $k$.
    \end{tabbing}}
\end{algorithm}

\begin{remark}
We emphasise that 
  the definition of the a posteriori indicators $\estT(S)$,
  $S\in\sides(\tria)$, usually involve the finite
  element solution on $\tria$. The evaluation of the indicators
  requires therefore the solution of the corresponding discrete
  system, which is computed in the step \textsf{SOLVE}. 

  The step \textsf{MARK} is the so called \emph{modified maximum
  marking strategy}. Roughly speaking, our modification of the maximum
  marking strategy replaces the role of the error indicator
  associated with an edge $S\in\sides(\tria_k)$ by the sum of the error indicators over
  the edges $\refd{\tria_k}{S}$, which necessarily 
  have to be bisected together with $S$ in order to retain a conforming triangulation. 
  When more then one edge is marked for refinement the effect of
  previously marked edges has to be subtracted. This is encoded in the
  set $\widetilde{\mathcal{M}}_k$. 
\end{remark}

Under the assumptions in Section~\ref{sec:energy}, it is proved in
\cite{DieKreuStev:15} that this AFEM is instance
optimal.
\begin{theorem}\label{thm:IO}
 Assume that \eqref{eq:emon}, \eqref{eq:errbnd}, and \eqref{eq:LDE}
  are satisfied and
  let $\{\tria_k\}_{k \in \mathbb{N}_0}$ 
  be the sequence of
  triangulations in 2D produced by
  Algorithm~\ref{algo:AFEM}.
  Then, there exists a constant $C\ge 1$, 
  such 
  that 
  \begin{align*}
    \energy(\tria_k)\le \energy(\tria)\qquad\text{for
      all}~\tria\in\bbT~\text{with}\quad C\#(\tria\setminus\tria_\bot)\le \#(\tria_k\setminus\tria_\bot).     
  \end{align*}
  We have $C \Cleq 1/\mu^2$ depending on the constants
    in~\eqref{eq:errbndb} and~\eqref{eq:LDE}.
\end{theorem}
\begin{proof}
  For a proof see \cite[Theorem 7.3]{DieKreuStev:15}.
\end{proof}

\begin{remark}
  The proof of instance optimality
  in \cite{DieKreuStev:15} relies on \emph{fine properties} of populations 
  \cite[\S6]{DieKreuStev:15}
  (see Remark~\ref{rem:Pop}).
  These technical arguments are only used to conclude 
  Theorem~\ref{thm:IO}
  from~\eqref{eq:emon},~\eqref{eq:errbnd}
  and~\eqref{eq:LDE}. Therefore, we neglect them and refer the
  interested reader to \cite{DieKreuStev:15}.
  
  A generalisation of populations and, hence, of Theorem~\ref{thm:IO}
 from two to higher dimensions is open.
\end{remark}

\section{Poisson Problem}
\label{sec:poisson}
In this section we show how a standard Crouzeix-Raviart finite
element method for the Poisson problem fits into the
framework of Section~\ref{sec:abstract-setting}. 

For $f\in L^2(\Omega)$,
find the unique weak solution $u\in
H_0^1(\Omega)$ of
\begin{align}\label{eq:Poisson}
  -\Delta u = f\quad\text{in}~\Omega\qquad\text{and}\qquad u=0\quad\text{on}~\partial\Omega.
\end{align}
Here $\Omega\subset \setR^2$ is a bounded polygonal domain. For any
open subset  $\omega\subset\Omega$, we shall denote by $H^1(\omega)$ the 
usual Sobolev space of functions in $L^2(\omega)$ whose first
derivatives are also in $L^2(\omega)$. Denoting the norm on
$L^2(\omega)$ by $\norm[\omega]{\cdot}$, then, thanks to Friedrichs
inequality, we have that $\norm[\Omega]{\nabla\cdot}$ is a norm on the space
$H_0^1(\Omega)$ of functions in $H^1(\Omega)$ with zero trace on
$\partial\Omega$.

\subsection{Crouzeix-Raviart Finite Element Framework}
\label{sec:CR-Framework}
We shall first introduce the Crouzeix-Raviart discretisation
of the Poisson problem and then provide some results, which are crucial 
in order to verify the assumptions needed to apply
  the abstract framework of
Section~\ref{sec:energy}. 

For $\tria\in\bbT$, we denote by $P_k(T)$ the space of polynomials of
degree at most $k$ on $T\in\tria$.  The $P_1$ nonconforming finite element space 
of Crouzeix and Raviart is defined by
\begin{align*}
  \CR:=\big\{\vT\in L^2(\Omega)\colon 
&\vT\vert_T\in P_1(T)~\text{a.e. in}~T\in\tria,\\
&\text{and}~\int_S\jump{\vT}\ds =0~\text{for all}~
S\in\sides(\tria)\big\}. 
\end{align*}
Here and throughout the paper, $\jump{\vT}|_S:=(\vT|_{T_1}-\vT|_{T_2})|_S$ denotes the jump of
$\vT$ across $S$ when $S=T_1\cap T_2\in\sides(\tria)$, and
$\jump{\vT}|_S=\vT|_S$ if
$S\in\sides(\tria)$ and $S\subset\partial\Omega$. Note that on interior edges this definition 
is only unique up to its sign, i.e., up to the choice of $T_1$ and
$T_2$. However, we shall only use it to describe orientation independent
properties. We define the subset of continuous 
functions by $\VoT:=\CR\cap H^1_0(\Omega)$.

The discrete Poincar\'e inequality for piecewise
$H^1$ functions \cite{Brenner:2004,BrennerScott:08}
guarantees that $\int_\Omega \gradnc^\tria u_\tria\cdot \gradnc^\tria v_\tria\dx $,
$u_\tria,v_\tria\in\CR$, is a scalar product.
Here the nonconforming or piecewise gradient is defined  by
\begin{align*}
  (\gradnc^\tria \vT)|_T:=\nabla({\vT}|_T)\quad\text{for all}~T\in\tria.
\end{align*}
We have that $\gradnc^{\tria_\star}|_{\CR}=\gradnc^{\tria}|_{\CR}$ when
$\tria_\star\ge \tria$. In situations when there is no danger of
confusion, we shall skip
the dependence on the triangulation and
simply write $\gradnc$. 
Thanks to the Riesz representation theorem, the 
Crouzeix-Raviart approximation $u_\tria\in\CR$ to \eqref{eq:Poisson}
is uniquely defined by 
\begin{align}\label{eq:discreteproblem}
 \int_\Omega \gradnc u_\tria\cdot \gradnc v_\tria\dx
   = \int_\Omega f v_\tria\dx
   \quad\text{for all }v_\tria\in\CR.
\end{align}
Since 
\begin{align}\label{eq:conv}
  \inf_{\tria\in\bbT,\#\tria\le N}\norm{\gradnc (u-\uT)}\to 0 \quad\text{as}~N\to\infty
\end{align}
(compare e.g.\ with~\cite{Gudi:2010}),
we define $\uT[\tria^\top]:=u$.

Next, we 
introduce the nonconforming interpolation operator
\begin{align*}
  \ipol:\operatorname{span}
  \big(\{\CR[\tria_\star]:\tria_\star\ge\tria\}\cup\{H^1_0(\Omega)\}\big)\to\CR
\end{align*}
defined by
\begin{align}\label{def:ipol}
 \int_S\ipol v\ds :=  \int_S v\ds\qquad\text{for all}~S\in\sides(\tria).
\end{align}
Note that the right-hand side of ~\eqref{def:ipol} is always well-defined. 
Since $\ipol v$ is piecewise affine, we
may equivalently define $(\ipol v)(\midpoint{S}) =  \frac1{|S|}\int_S v\ds$,
where  $\midpoint{S}$ denotes the barycenter of $S\in\sides(\tria)$. Consequently,
this operator is well defined since it determines exactly the
degrees of freedom of the Crouzeix-Raviart function. It satisfies the
well known projection property \cite[Lemma~3.1]{Agouzal:1994}
\begin{align}\label{eq:mean}
  \gradnc(\ipol v)|_T=\frac1{\abs{T}}\int_T \gradnc v\dx,
\end{align}
which directly implies $\norm{\gradnc(v-\ipol v)}\le \norm{\gradnc(v-\vT)}$ for all $\vT\in\CR$.
Moreover, we have the following approximation and stability properties
\begin{align}\label{eq:stability}
  \frac{1}{\Lambda}\norm[T]{\hT^{-1}(v-\ipol v)}\leq \norm[T]{\gradnc(v-\ipol v)}\leq
  \norm[T]{\gradnc v},
\end{align}
with $0<\Lambda\leq 0.4396$ \cite[Theorem~2.1]{CarstensenGedicke:2014};
this is a consequence of a discrete Friedrichs inequality \cite{BrennerScott:08} 
and a scaling argument. Since $\gradnc\vT$ is piecewise constant on
$\tria$ for $\vT\in\CR$, we obtain as an immediate consequence of~\eqref{eq:mean} that 
\begin{align}\label{eq:mean2}
  \int_\Omega \gradnc \vT\cdot\gradnc v\dx =
  \int_\Omega\gradnc\vT\cdot\gradnc (\ipol v)\dx
\end{align}
for all $v\in H^1_0(\Omega)\cup \CR[\tria_\star]$ with $\tria_\star\geq\tria$.

In order to prove the equivalence in~\eqref{eq:errbndb},
we shall also need some operator $\transop_\tria^{\tria_\star}:\CR\to\CR[\tria_\star]$
for $\tria,\tria_\star\in\bbT$, $\tria\leq\tria_\star$.
Based on local averages, such an operator
was recently
introduced in \cite{CarstensenGallistlSchedensack:2013}. For our purposes, 
we need to slightly modify the construction of this so called
transfer operator in order to avoid contributions on unrefined edges
in the stability estimate of the following Theorem~\ref{t:errbndtransop}. The construction of
$\transop_\tria^{\tria_\star}$ and the proof of
Theorem~\ref{t:errbndtransop} is deferred to the appendix.

\begin{theorem}[transfer operator]\label{t:errbndtransop}
Let $\tria\in\bbT$ and $\tria_\star\in\widehat\bbT$ with $\tria\le\tria_\star$. 
There exists an operator
$\transop_\tria^{\tria_\star}:\CR\to\CR[\tria_\star]$ with 
\begin{multline*}
  \norm{\hT^{-1}(\vT - \transop_\tria^{\tria_\star} \vT)}^2
  + \norm{\gradnc(\vT - \transop_\tria^{\tria_\star} \vT)}^2\\
  \Cleq     \sum_{S\in\sides(\tria)\setminus\sides(\tria_\star)}
      h_S \;\|\jump{\gradnc \vT \cdot\tangente}\|_{S}^2
\end{multline*}
for all $\vT\in\CR$.
Here $\tangente$ denotes a tangential unit vector of
$S$. Consequently, $\jump{\gradnc \uT\cdot\tangente}\vert_S$ denotes the
tangential jump of $\gradnc \uT$ across 
interior sides $S\in\sides(\tria)$ with $S\not\subset\partial\Omega$
and the trace of the tangential derivative of $\uT$ for 
boundary sides $S\subset \partial\Omega$.
\end{theorem}

Finally, we introduce a
linear operator $\compop:\CR\to H^1_0(\Omega)$ from
\cite[Proposition~2.3]{CarstensenGallistlSchedensack:2014}. 
We shall use this operator in the proof of the lower diamond
estimate~\eqref{eq:LDE} in Section~\ref{sec:LDE} below. 
The construction is based on an interpolation into the 
space $\VoT\subset H_0^1(\Omega)$ of piecewise affine continuous
functions. Additional piecewise
continuous quadratics are then used to
guarantee some extra local mean property for the gradient.

\begin{lemma}\label{l:compop}
Let $\tria\in\bbT$. There exists a linear operator $\compop:\CR\to
H^1_0(\Omega)$ satisfying a
local gradient mean property 
\begin{align*}
  \int_T \nabla \compop v_\tria\dx = \int_T \gradnc v_\tria\dx
  \qquad\text{for all }T\in\tria,~\vT\in\CR
\end{align*}
and the approximation and stability property
\begin{align*}
 \norm{h_\tria^{-1}(v_\tria - \compop v_\tria)}
  + \norm{\gradnc(v_\tria - \compop v_\tria)}
  \Cleq \norm{\gradnc v_\tria}.
\end{align*}
\end{lemma}
\begin{proof}
  Compare with Step~2 in the proof of  \cite[Proposition~2.3]{CarstensenGallistlSchedensack:2014}. 
\end{proof}

Let $\tria_\star,\tria\in\bbT$ with $\tria\le\tria_\star$, then 
$\ipol[\tria_\star]\circ \compop:\CR\to\CR[\tria_\star]$ and this 
composed operator has the following properties.
\begin{lemma}\label{l:comprop}
Let $\tria_\star,\tria\in\bbT$ with $\tria\le\tria_\star$. Then for
$\vT\in\CR$ the
operator $\ipol[\tria_\star]\circ \compop:\CR\to\CR[\tria_\star]$ satisfies
the local gradient mean property
\begin{align*}
 \int_T \gradnc (\ipol[\tria_\star]\circ \compop) v_\tria \dx
= \int_T \gradnc v_\tria\dx\qquad\text{for all}~T\in\tria,
\end{align*}
the conservation property
\begin{align*}
 (\ipol[\tria_\star]\circ \compop) v_\tria\vert_T = v_\tria\vert_T
 \qquad\text{for all}~T\in\tria\cap\tria_\star,
\end{align*}
and the approximation and stability property
\begin{align*}
  \norm{h_\tria^{-1} (v_\tria - (\ipol[\tria_\star]\circ \compop) v_\tria)}
  & + \norm{\gradnc (v_\tria - (\ipol[\tria_\star]\circ \compop)
    v_\tria)}
  \Cleq \norm{\gradnc v_\tria}.
\end{align*}
\end{lemma}

\begin{proof}
The local gradient mean property is a direct consequence of  \eqref{eq:mean}
and Lemma~\ref{l:compop}. 
This also implies
$(\ipol\circ\compop) =\id\vert_{\CR}$. 
Thanks to~\eqref{def:ipol}, we have $(\ipol[\tria_\star] v)\vert_T = (\ipol v)\vert_T$ 
for all $T\in\tria\cap\tria_\star$ and $v\in H^1_0(\Omega)$. This proves
the conservation property.
The approximation and stability property is a consequence of 
$ v_\tria - (\ipol[\tria_\star]\circ \compop) v_\tria=
 v_\tria - \compop v_\tria
    + \compop v_\tria - (\ipol[\tria_\star]\circ \compop) v_\tria$, 
a triangle inequality, 
the observation that  $h_\tria^{-1}\leq h_{\tria_\star}^{-1}$, and
the approximation and 
stability properties of $\ipol[\tria_\star]$ and $\compop$
in~\eqref{eq:stability} and 
Lemma~\ref{l:compop}. 
\end{proof}

\subsection{Energy}
\label{sec:emin}
For a fixed $\gamma>2\Lambda^2>0$, with $\Lambda$
from~\eqref{eq:stability}, we 
define the generalised energy $\energy:\widehat\bbT\to\setR$ by
\begin{align*}
\energy(\tria)&:=
 -\int_\Omega \frac12 \abs{ \gradnc \uT}^2- f \uT\dx+\gamma\norm{\hT f}^2
\intertext{when $\tria\in\bbT$ and}
\energy(\tria^\top)&:=-\int_\Omega \frac12\abs{\nabla \uT[\tria^\top]}^2- f \uT[\tria^\top]\dx.
\end{align*}
We emphasise that in the above definition of $\energy$, the
  Dirichlet energy of the Crouzeix-Raviart approximation appears with
  a negative sign. In fact, although the Crouzeix-Raviart finite element solution $\uT\in\CR$
    maximises $\vT\mapsto-\int_\Omega \frac12 \abs{ \gradnc \vT}^2- f \vT\dx$
    in $\CR$, this expression is in general not increasing under
    refinement. This is due to the fact, that 
    Crouzeix-Raviart spaces corresponding to nested triangulations 
    are not nested, i.e., $\tria\le\tria_\star$ does 
    not imply $\CR\subset\CR[\tria_\star]$.
    It was observed in
    \cite[Proof of Theorem 6.5]{CarstensenGallistlSchedensack:2013}, however, 
    that the energy $\energy$ is even monotonically \emph{decreasing}
    under refinement,  thanks to the term $\norm{\hT f}^2$. The
    following results states a refined version of this observation.

\begin{lemma}\label{l:ediff=err}
  Let $\tria\in\bbT$ and $\tria_\star\in\widehat\bbT$ with
  $\tria_\star\ge\tria$, then 
    \begin{align}\label{eq:ediffa}
     \begin{split}
       \energy(\tria)
       -\energy(\tria_\star) \eqsim
       \norm{\gradnc(\uT[\tria_\star]-\uT)}^2+ 
       \norm[\Omega(\tria\setminus\tria_\star)]{\hT
         f}^2.
     \end{split}
    \end{align}
    Moreover, we have
    \begin{align*}
      \energy(\tria) -\energy(\tria^\top)\eqsim
      \norm{\gradnc(u-\uT)}^2+\osc(\tria)^2.
    \end{align*}
  Here the oscillation is defined as 
  \begin{align*}
    \osc(\tria)^2:=\sum_{T\in\tria}\int_T\hT^2\abs{f-f_T}^2\dx
    \quad\text{with}~f_T:=\frac1{\abs{T}}
    \int_Tf\dx. 
  \end{align*}
\end{lemma}

\begin{proof}
  In order to prove the first statement, we observe that 
\begin{align*}
   \frac12\norm{\gradnc(\uT[\tria_\star]-\uT)}^2&
      \\
      &\hspace{-2.5cm}=\frac12\int_\Omega\abs{\gradnc \uT[\tria_\star]}^2 
        -\abs{\gradnc\uT}^2\dx 
    -\int_\Omega\gradnc\uT\cdot\gradnc (\uT[\tria_\star]-\uT)\dx
\end{align*}
The integral mean property \eqref{eq:mean2} and 
the discrete problem \eqref{eq:discreteproblem} imply
\begin{align*}
 -\int_\Omega\gradnc\uT\cdot\gradnc (\uT[\tria_\star]-\uT)\dx
  = -\int_\Omega f(\ipol \uT[\tria_\star]-\uT)\dx.
\end{align*}
Therefore, we have
  \begin{align} \label{eq:binom}
    \begin{split}
      \frac12\norm{\gradnc(\uT[\tria_\star]-\uT)}^2
      &=\int_\Omega \frac12\abs{\gradnc \uT[\tria_\star]}^2- f\uT[\tria_\star]\dx
      -\int_\Omega \frac12\abs{\gradnc \uT}^2-f\uT\dx
      \\
      &\quad+  \int_\Omega f(\uT[\tria_\star]-\ipol \uT[\tria_\star])\dx
      \\
      &= \energy(\tria)-\energy(\tria_\star) 
        -\gamma (\norm[\Omega]{\hT
          f}^2-\norm[\Omega]{\hT[\tria_\star] f}^2)
      \\
      &\quad+  \int_\Omega f(\uT[\tria_\star]-\ipol \uT[\tria_\star])\dx.
    \end{split}
  \end{align} 
Recalling the mesh-size reduction
  property~\eqref{eq:meshsize-red}, we have for the penultimate term
  that
  \begin{align}\label{eq:ediff=errProof3}
    \frac12\norm[\Omega(\tria\setminus\tria_\star)]{\hT
        f}^2 \le\norm[\Omega]{\hT f}^2-\norm[\Omega]{\hT[\tria_\star]
        f}^2\le \norm[\Omega(\tria\setminus\tria_\star)]{\hT
        f}^2.
  \end{align}
  Thanks to the definition of $\ipol$ in~\eqref{def:ipol}, we have that
  $\ipol\uT[\tria_\star]|_T=\uT[\tria_\star]|_T$ for all
  $T\in\tria\cap\tria_\star$.  Therefore, we can bound the last term in \eqref{eq:binom} by
  \begin{align*}
    \Big|\int_\Omega f(\uT[\tria_\star]-\ipol
    \uT[\tria_\star])\dx\Big|
   &=\Big|\int_{\Omega(\tria\setminus\tria_\star)}
    f(\uT[\tria_\star]-\ipol \uT[\tria_\star])\dx\Big|
    \\
    &\le \Lambda^2\norm[\Omega(\tria\setminus\tria_\star)]{\hT
      f}^2+\frac1{4\Lambda^2} \norm{\hT^{-1}(\uT[\tria_\star]-\ipol 
\uT[\tria_\star])}^2.
  \end{align*}
  Here we used  a scaled Young's inequality in the last
  step.
Then \eqref{eq:stability} implies
\begin{equation}\label{eq:ediff=errProof2}
\begin{aligned}
  \Big|\int_\Omega f(\uT[\tria_\star]-\ipol
    \uT[\tria_\star])\dx\Big|
    \le \Lambda^2\norm[\Omega(\tria\setminus\tria_\star)]{\hT f}^2
      +\frac1{4} \norm{\gradnc(\uT[\tria_\star]-\ipol \uT[\tria_\star])}^2.
\end{aligned}
\end{equation}
Combining~\eqref{eq:binom},~\eqref{eq:ediff=errProof3}
and~\eqref{eq:ediff=errProof2} yields
\begin{align}\label{eq:1}
  \begin{split}
    \frac14\norm{\gradnc(\uT[\tria_\star]-\uT)}^2
&+\big(\frac\gamma2-\Lambda^2\big)
   \norm[\Omega(\tria\setminus\tria_\star)]{\hT f}^2
    \\
    &\quad\le \energy(\tria)-\energy(\tria_\star)
    \\
    &\le \frac34
    \norm{\gradnc(\uT[\tria_\star]-\uT)}^2+(\gamma+\Lambda^2)
    \norm[\Omega(\tria\setminus\tria_\star)]{\hT f}^2.
  \end{split}
\end{align}
This proves the first equivalence~\eqref{eq:ediffa} for 
$\gamma>2\Lambda^2$.

  In order to prove the second claim, 
  we observe that the term 
  $\norm{\hT f}^2$ corresponds to the element residual in a scaled
  $L^2$-norm. Therefore,
  we can use Verf\"urth's bubble function technique
  (see e.g. \cite[\S1.4.5]{Verfuerth:13}), 
  in order to obtain the efficiency estimate
  \begin{align*}
   \norm{\hT f}^2 \Cleq \norm{\gradnc(u-\uT)}^2+\osc(\tria)^2,
  \end{align*}
  and the asserted estimate is a consequence of the first claim. 
\end{proof}

\subsection{Error Bounds}
\label{sec:error-bounds}
We define local error indicators on sides of $\tria\in\bbT$ by
\begin{align*}
  \estT^2(S):=\norm[\omega_S^\tria]{\hT f}^2 
    + h_S \norm[S]{\jump{\gradnc \uT \cdot\tangente}}^2,\quad S\in\sides(\tria)
\end{align*}
and set 
\begin{align*}
  \estT^2(\tilde{\sides}):=\sum_{S\in\tilde{\sides}} \estT^2(S)\quad\text{for}~\tilde\sides\subset\sides(\tria).
\end{align*}

We have the following relation between the estimator and some quasi-error.
\begin{theorem}\label{t:errbnd}
For $\tria\in\bbT,$ $\tria_\star\in\widehat\bbT$ with $\tria\leq\tria_\star$, we have
\begin{align*}
 \norm{\gradnc (\uT - u_{\tria_\star})}^2 
   + \norm[\Omega(\tria\setminus\tria_\star)]{\hT f}^2
 \eqsim \estT^2(\sides(\tria)\setminus\sides(\tria_\star)).
\end{align*}
\end{theorem}

Before we turn to prove Theorem~\ref{t:errbnd}, we observe
that, as a consequence of Theorem~\ref{t:errbnd} and
  Lemma~\ref{l:ediff=err}, the error estimator is equivalent to the energy
difference.
\begin{corollary}\label{c:errbndenergy}
For $\tria\in\bbT$ and $\tria_\star\in\bbThat$ with $\tria\leq\tria_\star$, we have
\begin{align*}
 \energy(\tria)-\energy(\tria_\star) 
 &\eqsim \estT^2(\sides(\tria)\setminus\sides(\tria_\star)).
\end{align*}
\end{corollary}

\begin{proof}[Proof of Theorem~\ref{t:errbnd}]
The proof is split into two parts. To prove $\Cgeq$, we follow
  \cite[Lemma~5.3]{CarstensenGallistlSchedensack:2014} and assume 
without loss of generality, that $\tria_\star\in\bbT$. The case
$\tria_\star=\tria^\top$ is then a consequence of~\eqref{eq:conv}.
Let $S\in\sides(\tria)\setminus\sides(\tria_\star)$
and let $\phi_S$ be the piecewise affine
continuation of $\phi_S(\midpoint{S})=1$ and 
$\phi_S\vert_{\Omega\setminus\omega_S^\tria}\equiv0$. Note that if
$S\subset\partial\Omega$, then 
$\phi_S|_{S}\not\equiv0$ and thus $\phi_S\not\in\VoT$.
Since $\jump{\gradnc \uT\cdot\tangente}\vert_S$ is constant on $S$,
an integration by parts together with $\int_S\phi_S\ds=\frac12 h_S$ yields
\begin{align*}
 \frac12h_S^{1/2} \norm[S]{\jump{\gradnc \uT \cdot\tangente}}
 &= \abs{\int_S \phi_S \jump{\gradnc \uT\cdot\tangente} \ds}
 = \abs{\int_{\omega_S^\tria} \gradnc \uT\cdot \Curl \phi_S\dx}. 
\end{align*}
Piecewise integration by parts reveals the 
$L^2$ orthogonality $\Curl \phi_S\bot_{L^2(\Omega)} \gradnc\CR[\tria_\star]$ 
\cite[Theorem~4.1]{ArnoldFalk:1989}. This
leads for any $v_{\tria_\star}\in\CR[\tria_\star]$ to 
\begin{align*}
 \abs{\int_{\omega_S^\tria} \gradnc \uT\cdot \Curl \phi_S\dx}
 = \abs{\int_{\omega_S^\tria} \gradnc (\uT-v_{\tria_\star})\cdot \Curl \phi_S\dx}.
\end{align*} 
Combining the previous equalities with 
a Cauchy inequality and the scaling $\norm{\Curl\phi_S}\eqsim 1$
we obtain for $\vT[\tria_\star]=\uT[\tria_\star]$ that
\begin{align*}
  h_S^{1/2} \norm[S]{\jump{\gradnc \uT \cdot\tangente}}
  &\Cleq 
  \norm[\omega_S^\tria]{\gradnc(\uT - u_{\tria_\star})}.
\end{align*}

We turn to prove  $\Cleq$. 
Consider $\tria_\star\in\bbT$ and
let $v_{\tria_\star}
=\argmin\{\norm{\gradnc(\uT - w_{\tria_\star})}
 : w_{\tria_\star}\in\CR[\tria_\star]\}$.
This implies 
\begin{align}\label{eq:orthogonalityargmin}
 \int_\Omega \gradnc(\uT - v_{\tria_\star})\cdot\gradnc w_{\tria_\star}\dx=0
  \qquad\text{for all }w_{\tria_\star}\in\CR[\tria_\star].
\end{align}
Then the Pythagorean theorem reads
\begin{align}\label{eq:phytagoras}
  \norm{\gradnc(\uT - u_{\tria_\star})}^2
  = \norm{\gradnc(\uT - v_{\tria_\star})}^2
    + \norm{\gradnc(v_{\tria_\star} - u_{\tria_\star})}^2.
\end{align}
For the second term on the right-hand side,
we obtain thanks to \eqref{eq:orthogonalityargmin} and the discrete
problem \eqref{eq:discreteproblem} 
for $w_{\tria_\star}:=v_{\tria_\star} - u_{\tria_\star}\in\CR[\tria_\star]$, that
\begin{align}\label{eq:proofdRelPoisson}
  \begin{split}
    \norm{\gradnc(v_{\tria_\star} -
      u_{\tria_\star})}^2
    &= \int_\Omega \gradnc u_{\tria}\cdot \gradnc w_{\tria_\star}\dx -
    \int_\Omega f w_{\tria_\star}\dx
    \\
    &= \int_\Omega f (\ipol w_{\tria_\star} - w_{\tria_\star})\dx.
  \end{split}
\end{align}
Here we used property~\eqref{eq:mean2} of
$\ipol$ in the last step. 
Moreover, we have by \eqref{def:ipol} that
$\ipol w_{\tria_\star}= w_{\tria_\star}$
on $\Omega(\tria\cap\tria_\star)$ and thus 
\begin{align*}
 \int_\Omega f (\ipol w_{\tria_\star} - w_{\tria_\star})\dx
 \Cleq \norm[\Omega(\tria\setminus\tria_\star)]{\hT f}
     \norm{\gradnc w_{\tria_\star}}.
\end{align*}
For the first-term on the right-hand side of \eqref{eq:phytagoras},
we have by the definition of $v_{\tria_\star}$ and
Theorem~\ref{t:errbndtransop} that 
\begin{align*}
 \norm{\gradnc(\uT - v_{\tria_\star})}^2
 &\leq \norm{\gradnc (\uT - \transop_\tria^{\tria_\star} \uT)}^2
 \\
 &\Cleq \sum_{S\in\sides(\tria)\setminus\sides(\tria_\star)}
    h_S \norm[S]{\jump{\gradnc \uT \cdot\tangente}}^2.
\end{align*}
A combination of the above
bounds with \eqref{eq:phytagoras} proves the assertion.

The same arguments apply for $\tria_\star=\tria^\top$.
\end{proof}

\subsection{Lower Diamond Estimates}
\label{sec:LDE}
For the conforming method in \cite{DieKreuStev:15}, the lower diamond
estimate \eqref{eq:LDE} is a consequence of the fact that the finite
element solutions are best approximations together with properties of
a locally defined stable quasi-interpolation operator. For non-conforming Crouzeix-Raviart
solutions of~\eqref{eq:discreteproblem}, however, we have a best approximation
property only in the following sense.

\begin{lemma}[quasi best approximation]\label{l:bestapprox}
For $\tria,\tria_\star\in\bbT$ with
$\tria\leq\tria_\star$, we have 
\begin{align*}
 \norm{\gradnc( u_{\tria_\star} - \uT)}
    \Cleq \min_{v_\tria\in\CR} &\norm{\gradnc(u_{\tria_\star} -v_\tria)}
      +\norm[\Omega(\tria\setminus\tria_\star)]{\hT f}.
\end{align*}
\end{lemma}

\begin{proof}
The projection property of $\ipol$ and the Pythagorean theorem
imply
\begin{align*}
 \norm{\gradnc(u_{\tria_\star} - \uT)}^2
  = \norm{\gradnc(u_{\tria_\star} - \ipol u_{\tria_\star})}^2
    + \norm{\gradnc(\ipol u_{\tria_\star} - \uT)}^2.
\end{align*}
Since $\ipol u_{\tria_\star}$ is the best approximation of $u_{\tria_\star}$
with respect to $\norm{\gradnc\cdot}$ in $\CR$, it suffices to estimate the second term on the 
right-hand side. 
Using the abbreviation $\varphi_\tria:=\ipol u_{\tria_\star} - \uT\in\CR$,
we have by \eqref{eq:discreteproblem} and \eqref{eq:mean2} that
\begin{align}\label{eq:eq1inproofbestapprox}
\begin{aligned}
 \norm{\gradnc (\ipol u_{\tria_\star} - \uT) }^2
  &= \int_\Omega \gradnc \uT[\tria_\star] \cdot \gradnc \varphi_\tria\dx
     - \int_\Omega f\, \varphi_\tria\dx
\end{aligned}
\end{align}
and by $(\ipol[\tria_\star]\circ \compop) \varphi_\tria\in\CR[\tria_\star]$
and again \eqref{eq:discreteproblem}
\begin{align}
\begin{aligned}
 \int_\Omega \gradnc \uT[\tria_\star] \cdot \gradnc \varphi_\tria\dx 
  &    - \int_\Omega f\, \varphi_\tria\dx
  =
 \int_\Omega f\, ((\ipol[\tria_\star]\circ \compop) \varphi_\tria
  - \varphi_\tria)\dx
  \\ 
  &   + \int_\Omega \gradnc \ipol u_{\tria_\star}\cdot 
  \gradnc (\varphi_\tria - (\ipol[\tria_\star]\circ
  \compop)\varphi_\tria)\dx
  \\
  &
  + \int_\Omega \gradnc (\ipol u_{\tria_\star} - u_{\tria_\star} ) 
  \cdot \gradnc (\ipol[\tria_\star]\circ \compop)\varphi_\tria\dx .
\end{aligned}
\end{align}
Thanks to the approximation property
and the conservation property in
Lemma~\ref{l:comprop}, we can bound the first term on the 
right-hand side by
\begin{align*}
   \int_\Omega f\, ((\ipol[\tria_\star]\circ \compop) \varphi_\tria - \varphi_\tria)\dx
   \Cleq \norm[\Omega(\tria\setminus\tria_\star)]{\hT f} \norm{\gradnc \varphi_\tria}.
\end{align*}
The gradient mean property of $\ipol[\tria_\star]\circ\compop$ implies
that the second term on the right-hand side 
vanishes.
Finally, for the third term the stability of $\ipol[\tria_\star]\circ\compop$ yields
\begin{multline*}
 \int_\Omega \gradnc (\ipol u_{\tria_\star} - u_{\tria_\star} ) 
           \cdot \gradnc (\ipol[\tria_\star]\circ \compop)\varphi_\tria\dx\\
 \Cleq \norm{\gradnc (\ipol u_{\tria_\star} - u_{\tria_\star} )}
   \norm{\gradnc\varphi_\tria}.
\end{multline*}
Combining these bounds proves the assertion.
\end{proof}

\begin{remark}
Lemma~\ref{l:bestapprox} can be proved similarly by replacing $\ipol[\tria_\star]\circ \compop$
by the transfer operator $\transop_\tria^{\tria_\star}$.
However, the proof of a quasi best approximation property 
for the Stokes equations in Lemma~\ref{l:bestapproxStokes} below requires the additional integral mean property
of $\ipol[\tria_\star]\circ \compop$ from Lemma~\ref{l:comprop}.
\end{remark}

As a consequence of Lemma~\ref{l:bestapprox}, we have
a lower diamond estimate for the energy defined in Section~\ref{sec:emin}.
\begin{theorem}[lower diamond estimate]\label{t:LDEPMPenergy}
The generalised energy $\mathcal{G}$ from Section~\ref{sec:emin}
satisfies the lower diamond estimate~\eqref{eq:LDE}.
\end{theorem}

\begin{proof}
Let $(\tria^\wedge,\tria_\vee;\tria_1,\dots,\tria_m)$, $m\in\setN$, 
be a lower diamond in $\bbT$. The set
$\tria_\vee\setminus\tria_j$ denotes the elements in $\tria_\vee$, which are coarsened 
in $\tria_j$. By Definition~\ref{d:LD} we have that the corresponding areas
$\Omega(\tria_j\setminus\tria_\vee)=\Omega(\tria_\vee\setminus\tria_j)$
are pairwise disjoint, i.e., the necessary coarsenings from
$\tria_\vee$ to $\tria_j$ for each $j$ take place in different
regions of $\Omega$ and are independent of each other. 
Since  $\tria^\wedge$ is the finest common coarsening and $\tria_\vee$
is the coarsest common refinement of
$\tria_1,\ldots,\tria_m$, we have that $\tria^\wedge$ can be obtained
from $\tria_\vee$ by performing exactly the coarsenings from $\tria_\vee$ to
$\tria_j$ for all $j=1,\ldots,m$. This implies that 
$\bigcup_{j=1}^m (\tria_j\setminus\tria_\vee) 
= \tria^\wedge\setminus\tria_\vee$, where the union is disjoint.

Therefore, we obtain with $\ipol[\tria^\wedge]u_{\tria_\vee} =u_{\tria_\vee}$ on 
$\Omega(\tria^\wedge\cap\tria_\vee)$, that
\begin{align*}
 \norm{\gradnc
       (u_{\tria_\vee} - \ipol[\tria^\wedge] u_{\tria_\vee})}^2
 = \sum_{j=1}^m \norm[\Omega(\tria_j\setminus\tria_\vee)]{\gradnc (u_{\tria_\vee} - 
    \ipol[\tria^\wedge] u_{\tria_\vee})}^2.
\end{align*}
Thanks to $\tria_j\setminus\tria_\vee 
\subset \tria^\wedge$, we have 
$\ipol[\tria^\wedge]\uT[\tria_\vee]=\ipol[\tria_j]\uT[\tria_\vee]$  
on $\Omega(\tria_j\setminus\tria_\vee)$, whence
\begin{align*}
  \norm[\Omega(\tria_j\setminus\tria_\vee)]{\gradnc (u_{\tria_\vee} - 
    \ipol[\tria^\wedge] u_{\tria_\vee})}^2
  =  \norm[\Omega(\tria_j\setminus\tria_\vee)]{\gradnc (u_{\tria_\vee} - 
    \ipol[\tria_j] u_{\tria_\vee})}^2.
\end{align*}
Moreover, since $\tria_j\cap\tria_\vee
\subset \tria_\vee$, it follows that
$\ipol[\tria_j]\uT[\tria_\vee]=\uT[\tria_\vee]$ on
$\Omega(\tria_\vee\cap\tria_j)$.  Combining these observations we
arrive at 
\begin{align}\label{eq:LDEinterpolation}
  \norm{\gradnc
       (u_{\tria_\vee} - \ipol[\tria^\wedge] u_{\tria_\vee})}^2
  = \sum_{j=1}^m \norm{\gradnc (u_{\tria_\vee} - 
    \ipol[\tria_j] u_{\tria_\vee})}^2.
\end{align}
Again, since the union $\bigcup_{j=1}^m (\tria_j\setminus\tria_\vee)
=\tria^\wedge\setminus\tria_\vee$ is disjoint we have
\begin{align}\label{eq:LDEhf}
 \sum_{j=1}^m \norm[\Omega(\tria_j\setminus\tria_\vee)]{h_{\tria_j} f}^2
   = \norm[\Omega(\tria^\wedge\setminus\tria_\vee)]{h_{\tria^\wedge} f}^2.
\end{align}
Therefore,
Lemma~\ref{l:bestapprox} applied to $\tria^\wedge$, implies 
\begin{align}
  \begin{split}
    \norm{\gradnc (u_{\tria_\vee} - u_{\tria^\wedge})}^2
    &+\norm[\Omega(\tria^\wedge\setminus\tria_\vee)]{h_{\tria^\wedge} f}^2
    \\
    &\hspace{-10mm}\eqsim \norm{\gradnc (u_{\tria_\vee} - \ipol[\tria^\wedge]
      u_{\tria_\vee})}^2 +\norm[\Omega(\tria^\wedge\setminus\tria_\vee)]{h_{\tria^\wedge} f}^2
    \\
    &\hspace{-10mm}= \sum_{j=1}^m \Big(\norm{\gradnc (u_{\tria_\vee} -
      \ipol[\tria_j]\uT[\tria_\vee])}^2
    +\norm[\Omega(\tria_j\setminus\tria_\vee)]{h_{\tria_j} f}^2\Big)
    \\
    &\hspace{-10mm}\eqsim \sum_{j=1}^m \Big(\norm{\gradnc (u_{\tria_\vee} -
      \uT[\tria_j])}^2
    +\norm[\Omega(\tria_j\setminus\tria_\vee)]{h_{\tria_j} f}^2\Big).
  \end{split}
\end{align}
The equivalence of the discrete errors and the 
energy difference (Lemma~\ref{l:ediff=err}) yields the assertion.
\end{proof}

\subsection{Instance Optimality of AFEM for the Poisson Problem}
\label{sec:inst-opt-Poisson}
The following result is a direct consequence of Theorem~\ref{thm:IO}
and Lemma~\ref{l:ediff=err}.
\begin{theorem}\label{thm:IO-Poisson}
  Let $\{\tria_k\}_{k \in \mathbb{N}_0}$ 
  be the sequence of
  triangulations 
  produced by Algorithm~\ref{algo:AFEM} with the energy of
  Section~\ref{sec:emin} and the estimator of Section~\ref{sec:error-bounds}. 
  Then there exist constants $C,\tilde C\ge1 $, such 
  that 
  \begin{gather*}
    \norm{\gradnc(u-\uT[\tria_k])}^2+\osc(\tria_k)^2\le \tilde C
    \big(\norm{\gradnc(u-\uT)}^2+\osc(\tria)^2\big)
    \intertext{for all $\tria\in\bbT$ with}~C\#(\tria\setminus\tria_\bot)\le \#(\tria_k\setminus\tria_\bot).     
  \end{gather*}
  We have $C \Cleq 1/\mu^2$ depending on the constants
    in Corollary~\ref{c:errbndenergy} and Theorem~\ref{t:LDEPMPenergy}. The constant $\tilde C$ depends on 
  the constants in Lemma~\ref{l:ediff=err}.
\end{theorem}
\begin{proof}
  The preceding considerations (Lemma~\ref{l:ediff=err},
Corollary~\ref{c:errbndenergy},
 and Theorem~\ref{t:LDEPMPenergy}) verified the 
assumptions~\eqref{eq:emon}, \eqref{eq:errbnd}, and 
\eqref{eq:LDE} and therefore the conditions of 
  Theorem~\ref{thm:IO} and thus, the assertion follows from Lemma~\ref{l:ediff=err}. 
\end{proof}

\begin{remark}[non-constant diffusion]
Consider the problem $-\operatorname{div}A\nabla u = f$
with invertible piecewise constant (with respect to $\tria_\bot$)
$A\in L^{\infty}(\Omega;\mathbb{R}^{2\times 2}_{\text{sym}})$.
Then the generalised energy 
\begin{align*}
 \energy(\tria):=-\int_\Omega \frac{1}{2} \gradnc \uT\cdot A\gradnc \uT 
     - f \uT \dx 
    + \widetilde{\gamma} \norm{\hT f}^2
\end{align*}
for some $\widetilde{\gamma}$ depending on $A$ and 
$\Lambda$ from~\eqref{eq:stability}
is monotonically decreasing and satisfies the lower diamond 
estimate. 
The error indicator
\begin{align*}
  \estT^2(S):=\norm[\omega_S^\tria]{\hT f}^2 
    + h_S \norm[S]{\jump{\gradnc \uT \cdot\tangente}}^2,\quad S\in\sides(\tria)
\end{align*}
satisfies the two-sided error bounds~\eqref{eq:errbnd}.
The proofs of those properties follow in the same way 
as for the Poisson problem with the additional property 
\begin{align*}
 \norm[T]{\gradnc(\uT[\tria_\star] - \uT)}^{2}
  \eqsim \int_T \gradnc(\uT[\tria_\star] - \uT)\cdot 
             A\gradnc(\uT[\tria_\star] - \uT)\dx.
\end{align*}
Hence, the corresponding adaptive algorithm is instance optimal.

If $A$ is not piecewise constant, \eqref{eq:mean2}
would involve additional oscillation terms.
It is unclear how to handle those terms in the proof 
of instance optimality.
\end{remark}

\section{Stokes Problem}
\label{sec:stokes}
One advantage of a Crouzeix-Raviart discretisation of the Stokes
problem is that the velocity approximation is piecewise
exactly divergence free, whence the discrete problem corresponds to
an energy minimising problem. 
We use this fact in order to show that the discrete setting fits into 
the framework of Section~\ref{sec:abstract-setting} thereby yielding
instance optimality of an AFEM with modified maximum strategy
(Algorithm~\ref{algo:AFEM}). The used techniques are similar
to the techniques in Section~\ref{sec:poisson} for the Poisson problem. 
We therefore sketch the proofs and only present the differences in detail.

Let  $\Omega\subset\setR^2$ be a bounded polygonal domain. 
For $\bff\in L^2(\Omega;\setR^2)$, we seek $\bfu\in H^1_0(\Omega;\mathbb{R}^2)$
and $p\in L^2_0(\Omega):=\{q\in L^2(\Omega):\int_\Omega q\dx =0\}$
such that
\begin{align}\label{eq:stokes}
  \begin{alignedat}{2}
    \int_\Omega \nabla \bfu : \nabla \bfv\dx + \int_\Omega p
    \operatorname{div} \bfv\dx &= \int_\Omega \bff \cdot\bfv\dx &\quad&\text{for all
    }\bfv\in H^1_0(\Omega;\mathbb{R}^2),
    \\
    \int_\Omega q \operatorname{div} \bfu\dx &= 0 &\quad&\text{for all
    }q\in L^2_0(\Omega)
  \end{alignedat}
\end{align}
with the scalar product $A:B:=\sum_{j,k=1}^2 A_{jk} B_{jk}$.
The inf-sup condition 
\begin{align}\label{eq:infsup}
  \norm{q}\Cleq \sup_{\bfv\in H^1_0(\Omega;\mathbb{R}^2)} 
     \frac{\int_\Omega q \operatorname{div} \bfv\dx}{\norm{\nabla \bfv}}
  \qquad\text{ for all }q \in L^2_0(\Omega)
\end{align}
guarantees the existence of a unique solution to \eqref{eq:stokes};
see e.g.~\cite{GiraultRaviart:86,BrezziFortin:91}.

\subsection{Crouzeix-Raviart Finite Element Framework.}
\label{sec:cr-stokes}
For $\tria\in\bbT$ denote the 
discrete velocity space by $\Vcr:=\CR\times\CR$ and the
discrete pressure space by $\QT:=\{q_\tria\in
L^2(\Omega):q\vert_T\in P_0(T),~T\in\tria\}$.
The approximation $\buT\in\Vcr$ and $p_\tria\in
\QT\cap L^2_0(\Omega)$ of \eqref{eq:stokes} is then given by
\begin{align}\label{eq:discreteStokes}
  \begin{split}
    \int_\Omega \gradnc \buT : \gradnc \bvT\dx + \int_\Omega p_\tria
    \divnc \bvT\dx &= \int_\Omega \bff\cdot \bvT\dx,
    \\
    \int_\Omega q_\tria \divnc \buT\dx &= 0
  \end{split}
\end{align}
for all $\bvT\in \Vcr$ and all $q_\tria\in \QT\cap L^2_0(\Omega)$.
The piecewise divergence is defined as 
$\divnc \bvT= \operatorname{tr} \gradnc\bvT \in \QT\cap
L^2_0(\Omega)$ (with the trace operator $\operatorname{tr}(A)=A_{11}+A_{22}$), whence $\divnc \buT=0$
pointwise. The discrete counter part 
\begin{align}\label{eq:dinfsup}
  \norm{q_\tria}\Cleq \sup_{\bvT\in\Vcr} \frac{\int_\Omega q_\tria\divnc \bvT\dx}{\norm{\gradnc \bvT}}
  \quad\text{ for all }q_\tria\in \QT\cap L^2_0(\Omega)
\end{align}
of~\eqref{eq:infsup}
guarantees the unique existence of a solution to
\eqref{eq:discreteStokes}; see~\cite{CrouzeixRaviart:73}.
As in Section~\ref{sec:poisson}, we define  $\buT[\tria^\top]:=\bfu$
and $p_{\tria^\top}:=p$.

The operators $\ipol$, $\transop_\tria^{\tria_\star}$, and $\compop$
can be generalised to two dimensions by a component-wise
application thereby maintaining the properties presented in
Section~\ref{sec:CR-Framework}. 
Since there is no danger of confusion,
we adopt the notation for scalar functions,
e.g.  the estimate~\eqref{eq:stability} in
this context is understood with vector valued functions.

\subsection{Energy}
\label{sec:eminStokes}
Let $\gamma>2\Lambda^2>0$.
Similarly as in Section~\ref{sec:emin}, we define the generalised
energy by 
\begin{align*}
 \energy(\tria):=-\int_\Omega \frac12 \abs{\gradnc \buT}^2 - \bff\cdot \buT\dx
      + \gamma\norm{h_\tria \bff}^2
\end{align*}
when $\tria\in\bbT$ and 
\begin{align*}
 \energy(\tria^\top):=-\int_\Omega \frac12 \abs{\nabla \buT[\tria^\top]}^2 - \bff\cdot\buT[\tria^\top]\dx.
\end{align*}

The following lemma bounds the pressure error by
the error of the velocity plus a suitable volume term;
see also \cite[Remark~3.2]{DariDuranPadra:95}.
This allows to consider the error of the velocity 
as the essential part in the analysis below.
\begin{lemma}\label{l:boundpStokes}
For $\tria\in\bbT$, $\tria_\star\in\bbThat$ with $\tria\leq\tria_\star$ we have
\begin{align*}
 \norm{p_{\tria_\star} - p_{\tria} }
  \Cleq \norm{\gradnc(\buT[\tria_\star] -\buT)}
    + \norm[\Omega(\tria\setminus\tria_\star)]{h_\tria \bff}.
\end{align*}
\end{lemma}

\begin{proof}
Assume that $\tria_\star\in\bbT$. Thanks to the discrete inf-sup condition~\eqref{eq:dinfsup}, there
exists $\bvT[\tria_\star]\in\Vcr[\tria_\star]$ with
$\norm{\gradnc \bvT[\tria_\star]}=1$ and
\begin{align}\label{eq:boundpStokes1}
  \norm{p_{\tria_\star} - p_\tria}
  \Cleq \int_\Omega (p_{\tria_\star} - p_\tria) \divnc \bvT[\tria_\star]\dx.
\end{align}
The gradient mean property \eqref{eq:mean} (compare also with~\eqref{eq:mean2})
and the discrete problem \eqref{eq:discreteStokes}
applied to $\buT[\tria_\star]$ and $p_{\tria_\star}$ with 
test function $\bvT[\tria_\star]$ as well as 
$\uT$ and $p_\tria$ with test function $\ipol \bvT[\tria_\star]$ lead to 
\begin{align*}
  &\int_\Omega (p_{\tria_\star} - p_\tria) \divnc \bvT[\tria_\star]\dx\\
  &\qquad\qquad = \int_\Omega \bff\cdot (\bvT[\tria_\star] - \ipol \bvT[\tria_\star])\dx
   - \int_\Omega \gradnc (\buT[\tria_\star] - \buT)\cdot\gradnc \bvT[\tria_\star]\dx.
\end{align*}
Since $(\ipol \bvT[\tria_\star])\vert_T = \bvT[\tria_\star]\vert_T$
for all $T\in\tria\cap\tria_\star$, the assertion follows from a
Cauchy inequality and the approximation
property \eqref{eq:stability} of $\ipol$.
For $\tria_\star=\tria^\top$, the assertion follows analogously.
\end{proof}

We next relate the energy difference to the error; compare
also with Lemma~\ref{l:ediff=err}.

\begin{lemma}\label{l:ediff=errST}
Let $\tria\in\bbT$ and $\tria_\star\in\widehat\bbT$ with
  $\tria_\star\ge\tria$, then 
    \begin{align*}
     \begin{split}
       \energy(\tria)
       -\energy(\tria_\star) \eqsim
       \norm{\gradnc(\buT[\tria_\star]-\buT)}^2+
       \norm[\Omega(\tria\setminus\tria_\star)]{\hT
         \bff}^2.
     \end{split}
    \end{align*}
    Moreover, we have
    \begin{align*}
      \energy(\tria) -\energy(\tria^\top)\eqsim
      \norm{\gradnc(\bfu-\buT)}^2+\norm{p-p_\tria}^2+\osc(\tria)^2,
    \end{align*}
  where the oscillation is defined as 
  \begin{align*}
    \osc(\tria)^2:=\sum_{T\in\tria}\int_T\hT^2\abs{\bff-\bff_T}^2\dx
    \quad\text{with}~\bff_T:=\frac1{\abs{T}}
    \int_T\bff\dx. 
  \end{align*}

\end{lemma}

\begin{proof}
Since $\int_T \divnc\ipol u_{\tria_\star} \dx =\int_T \divnc u_{\tria_\star}\dx =0$
for all $T\in\tria$, the proof of the first claim 
is a straight forward generalisation of the first statement in
Lemma~\ref{l:ediff=err}.

Verf\"urth's bubble function technique 
leads to the efficiency of the volume residual for the Stokes equations
\begin{align*}
 \norm{h_\tria \bff}^2 
  \Cleq \norm{\gradnc(\bfu-\buT)}^2 + \norm{p-p_\tria}^2 + \osc(\tria)^2;
\end{align*}
see \cite[Theorem 3.2]{DariDuranPadra:95}.
This, Lemma~\ref{l:boundpStokes} and the first claim proves the Lemma.
\end{proof}

\subsection{Error Bounds}
\label{sec:errbndStokes}

For $\tria\in\bbT$ we define 
local error indicators on  sides $S\in\sides(\tria)$ by 
\begin{align*}
 \estT^2(S):=\norm[\omega_S^\tria]{h_\tria \bff}^2 
    + h_S \norm[S]{\jump{\gradnc \buT\tangente}}^2
    \quad\text{and set}\quad \estT^2(\widetilde{\mathcal{S}}) :=\sum_{S\in \widetilde{\mathcal{S}}} \estT^2(S)
\end{align*}
for $\widetilde{\mathcal{S}}\subset\sides(\tria)$.
The following theorem states the equivalence of the error estimator to  
the so-called quasi-error; cf.\ Theorem~\ref{t:errbnd}.
\begin{theorem}\label{t:errbndST}
Any $\tria\in\bbT$, $\tria_\star\in\bbThat$ with $\tria\leq\tria_\star$ satisfy
\begin{align*}
 \norm{\gradnc(\buT-\buT[\tria_\star])}^2 
   + \norm[\Omega(\tria\setminus\tria_\star)]{h_\tria \bff}
  \eqsim \estT^2(\sides(\tria)\setminus\sides(\tria_\star)).
\end{align*}
\end{theorem}

As in Section~\ref{sec:error-bounds}, we observe the following 
equivalence, which is an immediate consequence of 
Theorem~\ref{t:errbndST} and Lemma~\ref{l:ediff=errST}.

\begin{corollary}\label{c:errbndenergyST}
For $\tria\in\bbT$ and $\tria_\star\in\bbThat$ with $\tria\leq\tria_\star$,
we have 
\begin{align*}
 \energy(\tria)-\energy(\tria_\star)
 \eqsim \estT^2(\sides(\tria)\setminus\sides(\tria_\star)).
\end{align*}
\end{corollary}

\begin{proof}[Proof of Theorem~\ref{t:errbndST}]
The claim $\Cgeq$ follows analogously to the
corresponding claim in Theorem~\ref{t:errbnd}.

In order to prove the converse direction $\Cleq$, consider first $\tria_\star\in\bbT$ and define
$\bvT[\tria_\star]:=\argmin\{\norm{\gradnc(\buT - \bwT[\tria_\star])}
 : \bwT[\tria_\star]\in\Vcr[\tria_\star]\}$. As in the proof 
of Theorem~\ref{t:errbnd} it suffices to bound $ \norm{\gradnc(\bvT[\tria_\star] -
      \buT[\tria_\star])}^2$; see also \eqref{eq:phytagoras}.
Note that $\divnc \bvT[\tria_\star]\neq 0$ in general, whence
for $\bwT[\tria_\star]:=\bvT[\tria_\star] -
\buT[\tria_\star]\in\Vcr[\tria_\star]$, we have
\begin{align*}
    \norm{\gradnc(\bvT[\tria_\star] -
      \buT[\tria_\star])}^2
    &= \int_\Omega \gradnc \buT\cdot \gradnc \bwT[\tria_\star]\dx -
    \int_\Omega \bff\cdot\bwT[\tria_\star]-
    p_{\tria_\star}\divnc \bwT[\tria_\star]\dx 
    \\
    &= \int_\Omega \bff\cdot (\ipol \bwT[\tria_\star] -
    \bwT[\tria_\star])\dx+\int_\Omega (p_{\tria_\star}-p_{\tria})\divnc \bwT[\tria_\star]\dx
\end{align*}
instead of~\eqref{eq:proofdRelPoisson}.
Since $\divnc \bwT[\tria_\star]=\divnc \bvT[\tria_\star]
= \divnc (\bvT[\tria_\star] - \buT)$,
a scaled Young inequality with $\varepsilon>0$ yields
\begin{align*}
 \int_\Omega (p_{\tria_\star} - p_\tria)\divnc \bwT[\tria_\star]\dx
  \leq \frac\varepsilon2 \norm{ p_{\tria_\star} - p_\tria}^2
   + \frac1{2\varepsilon}\norm{\gradnc(\bvT[\tria_\star] - \buT)}^2 .
\end{align*}
Bounding the pressure difference on the right-hand side with the help
of Lemma~\ref{l:boundpStokes}, we obtain for 
sufficiently small $\varepsilon>0$, that
\begin{align*}
 &\int_\Omega (p_{\tria_\star} - p_\tria)\divnc \bwT[\tria_\star]\dx\\
  &\qquad\leq \norm{\gradnc(\buT[\tria_\star] - \buT)}^2/2
   + C(\norm[\Omega(\tria\setminus\tria_\star)]{h_\tria \bff}^2
   + \norm{\gradnc(\bvT[\tria_\star] - \buT)}^2)
\end{align*}
for some $C=C(\varepsilon)>0$.
The first term can be absorbed on the left-hand side
of~\eqref{eq:phytagoras} and the assertion follows as in
the proof of Theorem~\ref{t:errbnd}.

For $\tria_\star=\tria^\top$ the claim follows analogously.
\end{proof}

\subsection{Lower Diamond Estimates}
\label{sec:LDEstokes}
In this section, we aim to prove the lower diamond
estimate~\eqref{eq:LDE} for the energy defined in
Section~\ref{sec:eminStokes}. In contrast to the lower diamond
estimate for the Poisson problem in Section~\ref{sec:LDE}, 
 the quasi best approximation property involves
the pressure.
\begin{lemma}[quasi best approximation]\label{l:bestapproxStokes}
Let $\tria,\tria_\star\in\bbT$ with $\tria\leq\tria_\star$,
then
\begin{align*}
 \norm{\gradnc(\buT[\tria_\star] - \buT)}
 &\Cleq \inf_{\bvT\in\Vcr} \norm{\gradnc (\buT[\tria_\star] - \bvT)}
      + \norm[\Omega(\tria\setminus\tria_\star)]{h_\tria \bff}
      \\
      &\quad+ \norm{p_{\tria_\star} - \Pi_\tria p_{\tria_\star}},
\end{align*}
where $\Pi_\tria:L^2(\Omega)\to L^2(\Omega)$ is the $L^2$-projection
onto $\QT$.
\end{lemma}

\begin{proof}
We recall from the proof of Lemma~\ref{l:bestapprox}, that it suffices
to bound the term $\norm{\gradnc (\ipol \buT[\tria_\star] - \buT)
}^2$. To this end note that for $\boldsymbol{\varphi}_\tria:=\ipol
\buT[\tria_\star] - \buT\in\Vcr$ we have $\divnc(\ipol[\tria_\star]\circ \compop)
\boldsymbol{\varphi}_\tria\not\equiv 0$ in general. Therefore, instead of 
\eqref{eq:eq1inproofbestapprox}, we obtain 
\begin{align*}
  \norm{\gradnc (\ipol \buT[\tria_\star] - \buT) }^2
  &= \int_\Omega \bff\cdot ((\ipol[\tria_\star]\circ \compop) \boldsymbol{\varphi}_\tria
  - \boldsymbol{\varphi}_\tria)\dx
  \\ 
  &\quad   + \int_\Omega \gradnc \ipol \buT[\tria_\star]\cdot 
  \gradnc (\boldsymbol{\varphi}_\tria - (\ipol[\tria_\star]\circ
  \compop)\boldsymbol{\varphi}_\tria)\dx
  \\
  &\quad
  + \int_\Omega \gradnc (\ipol \buT[\tria_\star] - \buT[\tria_\star] ) 
  \cdot \gradnc (\ipol[\tria_\star]\circ \compop)\boldsymbol{\varphi}_\tria\dx 
\\
&\quad 
    - \int_\Omega p_{\tria_\star} \divnc (\ipol[\tria_\star]\circ \compop) \boldsymbol{\varphi}_\tria\dx.
\end{align*}
The first three terms on the right-hand side can be estimated as in
the proof of Lemma~\ref{l:bestapprox}. For the last term, we obtain
with the gradient mean property and the stability property of
$\ipol[\tria_\star]\circ \compop$ (see 
Lemma~\ref{l:compop}), that
\begin{align*}
  \int_\Omega p_{\tria_\star} \divnc (\ipol[\tria_\star]\circ \compop) \boldsymbol{\varphi}_\tria\dx
  &= \int_\Omega (p_{\tria_\star} - \Pi_\tria p_{\tria_\star}) 
         \divnc (\ipol[\tria_\star]\circ \compop) \boldsymbol{\varphi}_\tria\dx\\
  &\Cleq \norm{p_{\tria_\star} - \Pi_\tria p_{\tria_\star}} \norm{\gradnc \boldsymbol{\varphi}_\tria}.
\end{align*}
This and the arguments of the proof of Lemma~\ref{l:bestapprox} imply the assertion.
\end{proof}

\begin{theorem}[lower diamond estimate]\label{t:LDE-ST}
The generalised energy $\mathcal{G}$ from \linebreak Section~\ref{sec:eminStokes}
satisfies the lower diamond estimate~\eqref{eq:LDE}.
\end{theorem}

\begin{proof}
Let $(\tria^\wedge,\tria_\vee;\tria_1,\dots,\tria_m)$, $m\in\setN$, 
be a lower diamond in $\bbT$.
We recall the definition of the $L^2$ projection in
Lemma~\ref{l:bestapproxStokes} and observe that $\Pi_{\tria^\wedge}
p_{\tria_\vee}\vert_T = p_{\tria_\vee}\vert_T$ when
$T\in\tria_\vee\cap\tria^\wedge$. Since the union
$\tria^\wedge\setminus\tria_\vee=\cup_{j=1}^m \tria_j\setminus\tria_\vee$
is  disjoint, we have
\begin{align*}
  \norm{p_{\tria_\vee} - \Pi_{\tria^\wedge} p_{\tria_\vee}} ^2
     = \sum_{j=1}^m \norm[\Omega(\tria_j\setminus\tria_\vee)]
     {p_{\tria_\vee} - \Pi_{\tria^\wedge} p_{\tria_\vee}}^2=\sum_{j=1}^m \norm{p_{\tria_\vee} - \Pi_{\tria_j} p_{\tria_\vee}}^2,
\end{align*}
where we have used 
$\Pi_{\tria^\wedge} p_{\tria_\vee}\vert_T = \Pi_{\tria_j} p_{\tria_\vee}\vert_T$
for $T\in\tria^\wedge\cap\tria_j$ in the last step.
This, \eqref{eq:LDEinterpolation} and \eqref{eq:LDEhf}
imply the equality
\begin{align*}
 &\norm{\gradnc(\buT[\tria_\vee] - \ipol[\tria^\wedge] \buT[\tria_\vee])}^2
  + \norm[\Omega(\tria^\wedge\setminus\tria_\vee)]{h_{\tria^\wedge} \bff}^2
  + \norm{p_{\tria_\vee} - \Pi_{\tria^\wedge} p_{\tria_\vee}}^2\\
 &\qquad = \sum_{j=1}^m \Big(\norm{\gradnc(\buT[\tria_\vee] - \ipol[\tria_j] \buT[\tria_\vee])}^2
  + \norm[\Omega(\tria_j\setminus\tria_\vee)]{h_{\tria_j} \bff}^2
  + \norm{p_{\tria_\vee} - \Pi_{\tria_j} p_{\tria_\vee}}^2\Big).
\end{align*}
The quasi best approximation property of Lemma~\ref{l:bestapproxStokes}
implies
\begin{align*}
  &\norm{\gradnc(\buT[\tria_\vee] - \buT[\tria^\wedge])}^2
  + \norm[\Omega(\tria^\wedge\setminus\tria_\vee)]{h_{\tria^\wedge} \bff}^2
  + \norm{p_{\tria_\vee} - \Pi_{\tria^\wedge} p_{\tria_\vee}}^2\\
 &\qquad \eqsim \sum_{j=1}^m \Big(\norm{\gradnc(\buT[\tria_\vee] - \buT[\tria_j])}^2
  + \norm[\Omega(\tria_j\setminus\tria_\vee)]{h_{\tria_j} \bff}^2
  + \norm{p_{\tria_\vee} - \Pi_{\tria_j} p_{\tria_\vee}}^2\Big).
\end{align*}
Lemma~\ref{l:boundpStokes} then yields  
\begin{align*}
 \norm{\gradnc (\buT[\tria_\vee]-\buT[\tria^\wedge])}^2
  &+ \norm[\Omega(\tria^\wedge\setminus\tria_\vee)]{h_{\tria^\wedge} \bff}^2\\
   &\eqsim \sum_{j=1}^m \Big(\norm{ \gradnc(\buT[\tria_\vee] - \buT[\tria_j])}^2
  + \norm[\Omega(\tria_j\setminus\tria_\vee)]{h_{\tria_j} \bff}^2\Big)
\end{align*}
and thus the assertion follows from Lemma~\ref{l:ediff=errST}.
\end{proof}

\subsection{Instance Optimality of AFEM for the Stokes Problem}
\label{sec:inst-opt-Stokes}
The following result is a direct consequence of Theorem~\ref{thm:IO}
and Lemma~\ref{l:ediff=errST}.
\begin{theorem}\label{thm:IO-Stokes}
  Let $\{\tria_k\}_{k \in \mathbb{N}_0}$ 
  be the sequence of
  triangulations 
  produced by Algorithm~\ref{algo:AFEM} with the energy of
  Section~\ref{sec:eminStokes} and the estimator of Section~\ref{sec:errbndStokes}. 
  There exist constants $C,\tilde C \ge1$, such 
  that 
  \begin{gather*}
    \norm{\gradnc(u-\uT[\tria_k])}^2+ \norm{p-p_{\tria_k}}^2 +\osc(\tria_k)^2\\
   \qquad\qquad\qquad\qquad\qquad\qquad\le \tilde C
    \big(\norm{\gradnc(u-\uT)}^2+ \norm{p-p_\tria}^2 +\osc(\tria)^2\big)
    \intertext{for all $\tria\in\bbT$ with}~C\#(\tria\setminus\tria_\bot)\le \#(\tria_k\setminus\tria_\bot).     
  \end{gather*}
  We have $C \Cleq 1/\mu^2$ depending on the constants
    in Corollary~\ref{c:errbndenergyST} and Theorem~\ref{t:LDE-ST}. The constant $\tilde C$ depends on 
  the constants in Lemma~\ref{l:ediff=errST}.
\end{theorem}
\begin{proof}
    The preceding considerations of Section~\ref{sec:stokes}
 (Lemma~\ref{l:ediff=errST}, Corollary~\ref{c:errbndenergyST}, and 
Theorem~\ref{t:LDE-ST}) verified the 
assumptions~\eqref{eq:emon}, \eqref{eq:errbnd}, and 
\eqref{eq:LDE} and therefore the conditions of
  Theorem~\ref{thm:IO} and thus, the assertion follows from Lemma~\ref{l:ediff=errST}. 
\end{proof}

  \providecommand{\bysame}{\leavevmode\hbox to3em{\hrulefill}\thinspace}
\providecommand{\MR}{\relax\ifhmode\unskip\space\fi MR }
\providecommand{\MRhref}[2]{
  \href{http://www.ams.org/mathscinet-getitem?mr=#1}{#2}
}
\providecommand{\href}[2]{#2}

\appendix

\section{A Modified Transfer Operator}
This appendix is concerned with the construction of the operator
$\transop_\tria^{\tria_\star}:\CR\to\CR[\tria_\star]$, $\tria \in\bbT$ and
$\tria_\star\in\widehat\bbT$ with $\tria\le\tria_\star$, and the proof
of 
Theorem~\ref{t:errbndtransop}. 
We recall that 
$\widehat{\tria}:=\refine(\tria;\sides(\tria))$ 
is the smallest admissible 
refinement of $\tria$ such that 
$\sides(\tria)\cap\sides(\widehat{\tria})=\emptyset$, i.e. all sides
of  $\tria$ 
are bisected in $\widehat\tria$.
Consequently, $\widehat{\tria}\wedge\tria_\star$ is
the smallest admissible refinement of $\tria$ such that 
$\sides(\widehat\tria)\cap(\sides(\tria)\setminus\sides(\tria_\star))=\emptyset$. 
This triangulation is the intermediate triangulation
from \cite{CarstensenGallistlSchedensack:2013}.

\begin{proposition}[intermediate triangulation]\label{p:propintermtria}
Let $T\in\tria$ and $\hat T\in\widehat{\tria}\wedge\tria_\star$ 
with $\hat T\subset T$, then
$h_{\hat T}\leq h_T\leq 4 h_{\hat T}$.
Moreover, we have
$\tria\cap\tria_\star=\tria\cap(\widehat{\tria}\wedge\tria_\star)$, 
$\sides(\tria)\cap\sides(\tria_\star)
=\sides(\tria)\cap\sides(\widehat{\tria}\wedge\tria_\star)$, and
$\sides(\tria)\setminus\sides(\widehat \tria\wedge\tria_\star)
=\sides(\tria)\setminus\sides(\tria_\star)$.
\end{proposition}

For $z\in\nodes(\widehat{\tria}\wedge\tria_\star)$, 
we call $(\widehat{\tria}\wedge\tria_\star)(z)
:=\{T\in\widehat{\tria}\wedge\tria_\star : z\in T\}$ the star of $z$
in $(\widehat{\tria}\wedge\tria_\star)$.  
Let $T\in (\widehat{\tria}\wedge\tria_\star)(z)$, then we define the
set of refined edge connected simplices in
$(\widehat{\tria}\wedge\tria_\star)(z)$ by 
\begin{align*}
  \mathcal{Z}(z;T):=\{&T'\in\widehat{\tria}\wedge\tria_\star : 
      \exists T=T_1,\dots,T_J=T'
      \in(\widehat{\tria}\wedge\tria_\star)(z), J\in\setN,\\
      & \text{such that}~T_j\cap T_{j+1}
          \in\sides(\tria_\star)\setminus
          \sides(\tria), j=1,\dots, J-1\}.
\end{align*}
Note that, apart from a refined analysis, the definition of
$\mathcal{Z}(z;T)$ is the main difference in
the construction of the transfer operator compared to 
\cite{CarstensenGallistlSchedensack:2013}.

We define an auxiliary operator 
$\av_\tria^{\tria_\star}:\CR\to \{v\in L^2(\Omega):v\vert_T\in P_1(T)
\text{ for all }T\in \widehat{\tria}\wedge\tria_\star\}$ by averaging
over the elements in $\mathcal{Z}(z;T)$. In particular,  
for $z\in\nodes(\widehat{\tria}\wedge\tria_\star)$ 
and $T\in(\widehat{\tria}\wedge\tria_\star)(z)$ we set
\begin{align*}
 \av_\tria^{\tria_\star} \vT\vert_T(z):=
 \begin{cases}
   \frac1{ \#\mathcal{Z}(z;T)}\sum_{T'\in\mathcal{Z}(z;T)} \vT\vert_{T'}(z),&\text{if}~z\in\Omega
   \\
   0,&\text{if}~z\in\partial\Omega.
 \end{cases}
\end{align*}
\begin{lemma}\label{l:avcontinuous}
We have that $\av_\tria^{\tria_\star} \uT$ is continuous on 
$\Omega\setminus\bigcup\{S\in\sides(\tria_\star)\cap\sides(\tria)\}$.
\end{lemma}

\begin{proof}
Let $T,T'\in(\widehat{\tria}\wedge\tria_\star)$
with $T\cap T'\not\in \sides(\tria_\star)\cap\sides(\tria)$,
then $\mathcal{Z}(z;T)=\mathcal{Z}(z;T')$ for $z\in \nodes(\tria)\cap T\cap T'$,
and therefore $\av_\tria^{\tria_\star} \uT\vert_T(z) 
= \av_\tria^{\tria_\star} \uT\vert_{T'}(z)$.
\end{proof}

The transfer operator then defines a
function $\transop_\tria^{\tria_\star}\vT\in\CR[\widehat\tria\cap\tria_\star]$ by
\begin{align}\label{eq:deftransop}
  (\transop_{\tria}^{\tria_\star} \vT)(\midpoint{S})=\begin{cases}
                   \vT(\midpoint{S}) 
                   &\text{if}~S\in
                   \sides(\tria_\star)\cap \sides(\tria),
                   \\
                   (\av_\tria^{\tria_\star} \vT)(\midpoint{S})&\text{if}~
                         S\in\sides(\tria_\star)
                         \setminus\sides(\tria).
                  \end{cases}
\end{align}

\begin{theorem}
We have that $\transop_\tria^{\tria_\star} \vT\vert_T = \vT\vert_T$
for all $T\in\tria\cap\tria_\star$ and
$\transop_\tria^{\tria_\star} \vT\in
\CR[\widehat{\tria}\wedge\tria_\star]\cap\CR[\tria_\star]$.
\end{theorem}

\begin{proof}
Proposition~\ref{p:propintermtria} and the definition 
of $\transop_\tria^{\tria_\star}$ imply 
$\transop_\tria^{\tria_\star} \vT\vert_T = \vT\vert_T$
for all $T\in\tria\cap\tria_\star$. Thanks to \eqref{eq:deftransop} 
we have $\transop_\tria^{\tria_\star} \vT\in
\CR[\widehat{\tria}\wedge\tria_\star]$.

Let
$S\in\sides(\tria_\star)\setminus\sides(\widehat{\tria}\wedge\tria_\star)$
and $S\subset \partial T$ for some $T\in\tria_\star$. 
Newest vertex bisection yields for $T\subset K\in\widehat{\tria}\wedge\tria_\star$, that
$\{S'\in\sides(\widehat{\tria}\wedge\tria_\star):S'\subset \partial K\}\cap\sides(\tria)=\emptyset$.
Therefore, by  \eqref{eq:deftransop}, we have
$(\transop_\tria^{\tria_\star} \vT)|_T=(\av_\tria^{\tria_\star}
\vT)|_T$. We can argue analogously for the other triangle adjacent to
$S$. Thus, thanks to Lemma~\ref{l:avcontinuous} we have that
$\transop_\tria^{\tria_\star} \vT$ is continuous along $S$. This
proves $\transop_\tria^{\tria_\star} \vT\in
\CR[\tria_\star]$.
\end{proof}
We are now in the position to prove Theorem~\ref{t:errbndtransop}.

\begin{proof}[Proof of Theorem~\ref{t:errbndtransop}]
Let $T\in\tria\setminus\tria_\star$ and $\hat{T}\in\widehat{\tria}\wedge\tria_\star$ 
with $\hat{T}\subset
T$ and let $\hat S$ be a side of $\hat T$. We denote
by $\psi_{\hat S}\in\CR[\widehat{\tria}\wedge\tria_\star]$ 
the Crouzeix-Raviart basis function defined by 
$\psi_{\hat S}(\midpoint{\hat S})=1$ and $\psi_{\hat S}(\midpoint{S})=0$ 
for $S\in\sides(\widehat{\tria}\wedge\tria_\star)\setminus\{\hat S\}$.
Then the  affine function $(\uT-\transop_\tria^{\tria_\star}
\uT)\vert_{\widehat{T}}$ reads as
\begin{multline*}
 (\uT - \transop_\tria^{\tria_\star} \uT)\vert_{\hat{T}}
 \\=
  \sum_{\hat S\in\sides(\tria_\star)\setminus\sides(\tria),\,
  \hat S\subset \hat T} 
    \big(\uT\vert_{\hat{T}}(\midpoint{\hat S}) 
     - (\av_\tria^{\tria_\star} \uT)\vert_{\hat{T}}(\midpoint{\hat
       S})\big)\psi_{\hat S}.
\end{multline*}
Thanks to the triangle inequality and 
 $\norm[\hat{T}]{\hT^{-1}\psi_{\hat S}}
\eqsim\norm[\hat{T}]{\gradnc\psi_{\hat S}}\eqsim 1$, we have
\begin{align*}
  &\norm[\widehat{T}]{\hT^{-1}(\uT - \transop_\tria^{\tria_\star} \uT)}
  + \norm[\widehat{T}]{\gradnc(\uT - \transop_\tria^{\tria_\star}
    \uT)}
  \\
  &\qquad\qquad
   \Cleq \sum_{\hat S\in\sides(\tria_\star)\setminus\sides(\tria),\,
  \hat S\subset \hat T} 
    \abs{ \uT\vert_{\hat{T}}(\midpoint{\hat S}) 
        - (\av_\tria^{\tria_\star}
        \uT)\vert_{\hat{T}}(\midpoint{\hat S})} .
\end{align*}
We recall from Lemma~\ref{l:avcontinuous}, that 
 $(\uT-\av_\tria^{\tria_\star} \uT)\vert_{\hat S}\in P_1(\hat S)$ is affine on $\hat S$,
 $\hat S=\mathrm{conv}\{y_1,y_2\}\in\sides (\tria_\star)\setminus\sides(\tria)$, and thus
 we have from the definition of $\av_\tria^{\tria_\star}$ that
\begin{multline*}
\abs{\uT\vert_{\hat{T}}(\midpoint{\hat S})
  -(\av_\tria^{\tria_\star} \uT)\vert_{\hat{T}}(\midpoint{\hat S})}\\
\leq\sum_{k=1}^2\frac{
   \abs{
    \sum_{T'\in\mathcal{Z}(y_k;\hat{T})}
       ( \uT\vert_{\hat{T}}(y_k)- \uT\vert_{T'}(y_k))
    }}{
 2\,\#\mathcal{Z}(y_k;\hat{T})}.
\end{multline*}
Fix $k\in\{1,2\}$, $T'\in\mathcal{Z}(y_k;\hat{T})$,
and let $T_1,\ldots,T_J\in\mathcal{Z}(y_k;\widehat{T})$,
with $J\in\mathbb{N}$, $\hat{T}=T_1$, $T_J=T'$, and 
$T_j\cap T_{j+1}\in\sides(\widehat{\tria}\wedge\tria_\star)\setminus\sides(\tria)$
for $j=1,\ldots,J-1$. Then by a telescopic sum argument, we have
\begin{align*}
 \uT\vert_{\hat{T}}(y_k) - \uT\vert_{T'}(y_k)
=\sum_{j=1}^{J-1}( \uT\vert_{T_j}(y_k) - \uT\vert_{T_{j+1}}(y_k)) .
\end{align*}
Let 
$S=T_j\cap
T_{j+1}\in\sides(\widehat{\tria}\wedge\tria_\star)\setminus\sides(\tria)$
for some $j\in\{1,\ldots,J-1\}$.
Since $\jump{\uT}\vert_{S}$ is affine, H\"older's inequality yields
\begin{align*}
 \abs{\jump{\uT}\vert_{S}(y_k)}^2
   = 4 h_{S}^{-2} \Big(\int_{S}\abs{\jump{\uT}}\dx\Big)^2
   \leq 4 h_{S}^{-1} \norm[{S}]{\jump{\uT}}^2.
\end{align*}
The shape regularity of $\bbT$ implies $\#\mathcal{Z}(y_k;\hat{T})\eqsim 1$.
Let $\sides(\tria)(z):=\{S\in\sides(\tria) : z\in S\}$ 
denote the edges of $\tria$ which contain $z\in\nodes(\tria)$.
Summing over $\hat{T}\in\widehat{\tria}\wedge\tria_\star$ with 
$\hat{T}\subset T$ and
a Poincar\'e inequality along 
$S\in\sides(\tria)(z)\setminus\sides(\tria_\star)$ shows
\begin{align}
  \begin{split}
    \norm[T]{\hT^{-1}(\vT - \transop_\tria^{\tria_\star} \vT)}^2
    &+ \norm[T]{\gradnc(\vT - \transop_\tria^{\tria_\star} \vT)}^2\\
    &\hspace{-3mm} \Cleq \sum_{z\in\nodes(\tria)\cap T}
    \sum_{S\in\sides(\tria)(z)\setminus\sides(\tria_\star)} h_S
    \;\|\jump{\gradnc \vT \cdot\tangente}\|_{S}^2.
  \end{split}
\end{align}
Here we used that $h_{S}\eqsim h_{S'}$
for all $\sides(\widehat\tria\wedge\tria_\star)\ni S\subset  S'\in
\sides(\tria)$; see Proposition~\ref{p:propintermtria}. 
Summing over all $T\in\tria$, and accounting for the finite overlap of 
$\sides(\tria)(z)\setminus\sides(\tria_\star)$, $z\in\nodes(\tria)$,
proves the assertion.
\end{proof}


\begin{thebibliography}{CKNS08}

\bibitem[Ago94]{Agouzal:1994}
A.~Agouzal, \emph{A posteriori error estimator for nonconforming finite element 
methods}, Appl. Math. Lett. \textbf{7} (1994), no.~5, 61--66.


\bibitem[AF89]{ArnoldFalk:1989}
D.\,N.~Arnold and R.\,S.~Falk, \emph{A uniformly accurate finite element method for the
{R}eissner-{M}indlin plate}, SIAM J. Numer. Anal. \textbf{26} (1989), no.~6, 1276--1290.

\bibitem[B{\"a}n91]{Baensch:91}
E. B{\"a}nsch, \emph{Local mesh refinement in 2 and 3 dimensions}, IMPACT
  Comput. Sci. Engrg. \textbf{3} (1991), 181--191.

\bibitem[BDD04]{BiDaDeV:04}
P. Binev, W. Dahmen, and R. DeVore, \emph{Adaptive finite element
  methods with convergence rates}, Numer. Math \textbf{97} (2004), 219--268.

\bibitem[BDK12]{BelenkiDieningKreuzer:12}
L.~Belenki, L.~Diening, and C.~Kreuzer, \emph{Quasi-optimality of an adaptive
  finite element method for the $p$-{L}aplacian equation}, IMA J. Numer. Anal.
  \textbf{32} (2012), no.~2, 484--510.

\bibitem[BF91]{BrezziFortin:91}
F. Brezzi and M. Fortin, \emph{Mixed and hybrid finite element
  methods}, Springer Series in Computational Mathematics 15, 1991.

\bibitem[Bin07]{Binev:07}
P.~Binev, \emph{Adaptive methods and near-best tree approximation}, Oberwolfach
  Reports, vol.~29, 2007, pp.~1669--1673.

\bibitem[BM11]{BeckerMao:11}
R. Becker and S. Mao, \emph{Quasi-optimality of adaptive nonconforming
  finite element methods for the {S}tokes equations}, SIAM J. Numer. Anal.
  \textbf{49} (2011), no.~3, 970--991.

\bibitem[BMS10]{BeckerMaoShi:10}
R. Becker, S. Mao, and Z. Shi, \emph{A convergent nonconforming
  adaptive finite element method with quasi-optimal complexity}, SIAM J. Numer.
  Anal. \textbf{47} (2010), no.~6, 4639--4659.

\bibitem[BN10]{BonitoNochetto:2010}
A. Bonito and R.\,H. Nochetto, \emph{Quasi-optimal convergence rate of
  an adaptive discontinuous {G}alerkin method}, SIAM J. Numer. Anal.
  \textbf{48} (2010), no.~2, 734--771.

\bibitem[BFH14]{BraesFraunholzHoppe:2004}
D. Braess, T. Fraunholz and R.\,H.\,W. Hoppe, \emph{An equilibrated a
  posteriori error estimator for the interior penalty discontinuous
  Galerkin method}, SIAM J. Numer. Anal. \textbf{52} (2014), no.~4, 2121--2136.

\bibitem[Bre03]{Brenner:2004}
S.\,C. Brenner, \emph{Poincar\'e-{F}riedrichs inequalities for piecewise
  {$H^1$} functions}, SIAM J. Numer. Anal. \textbf{41} (2003), no.~1, 306--324.

\bibitem[BS08]{BrennerScott:08}
S.\,C. Brenner and R. Scott, \emph{The mathematical theory of finite
  element methods}, Springer Texts in Applied Mathematics 15, 2008.

\bibitem[CFPP14]{CarstensenFeischlPagePraetorius:2014}
C.~Carstensen, M.~Feischl, M.~Page, and D.~Praetorius, \emph{Axioms of
  adaptivity}, Comput. Math. Appl. \textbf{67} (2014), no.~6, 1195--1253.

\bibitem[CGS]{CarstensenGallistlSchedensack:2014}
C. Carstensen, D. Gallistl, and M. Schedensack, \emph{Adaptive
  nonconforming {C}rouzeix-{R}aviart {FEM} for eigenvalue problems}, (accepted
  for publication in Math. Comp.).

\bibitem[CGS13]{CarstensenGallistlSchedensack:2013}
\bysame, \emph{Discrete reliability for {C}rouzeix--{R}aviart {FEM}s}, SIAM
  J. Numer. Anal. \textbf{51} (2013), no.~5, 2935--2955.

\bibitem[CG14]{CarstensenGedicke:2014}
C.~Carstensen and J. Gedicke, \emph{Guaranteed lower bounds for eigenvalues},
  Math. Comp. \textbf{83} (2014), 
  no.~290, 2605--2629.

\bibitem[CH06]{CarstensenHoppe:06b}
C. Carstensen and R.\,H.\,W. Hoppe, \emph{{Convergence analysis of an
  adaptive nonconforming finite element method.}}, Numer. Math. \textbf{103}
  (2006), no.~2, 251--266.

\bibitem[CHX]{ChHoXu:06}
L. Chen, M.\,J. Holst, and Jinchao Xu, \emph{Convergence and optimality
  of adaptive mixed finite element methods}, Math. Comp. \textbf{78} (2009), 
  no.~265, 35--53.

\bibitem[CKNS08]{CaKrNoSi:08}
J.\,M. Casc\'on, C.\, Kreuzer, R.\,H. Nochetto, and K.\,G.
  Siebert, \emph{Quasi-optimal convergence rate for an adaptive finite element
  method}, SIAM J. Numer. Anal. \textbf{46} (2008), no.~5, 2524--2550.

\bibitem[CPR13]{CarstensenPeterseimRabus:2013}
C. Carstensen, D. Peterseim, and H. Rabus, \emph{Optimal adaptive
  nonconforming {FEM} for the {S}tokes problem}, Numer. Math. \textbf{123}
  (2013), no.~2, 291--308.

\bibitem[CR73]{CrouzeixRaviart:73}
M.~Crouzeix and P.-A. Raviart, \emph{Conforming and nonconforming finite
  element methods for solving the stationary {S}tokes equations. {I}}, Rev.
  Fran\c caise Automat. Informat. Recherche Op\'erationnelle S\'er. Rouge
  \textbf{7} (1973), no.~R-3, 33--75.

\bibitem[DDP95]{DariDuranPadra:95}
E. Dari, R. Dur{\'a}n, and C. Padra, \emph{Error estimators for
  nonconforming finite element approximations of the {S}tokes problem}, Math.
  Comp. \textbf{64} (1995), no.~211, 1017--1033.
 
\bibitem[DKS15]{DieKreuStev:15}
L.~Diening, C.~Kreuzer, and R.~Stevenson, \emph{Instance optimality of the
  adaptive maximum strategy}, Found. Comput. Math. (2015), online first.



\bibitem[D{\"o}r96]{Doerfler:96}
W. D{\"o}rfler, \emph{A convergent adaptive algorithm for {P}oisson's
  equation}, SIAM J. Numer. Anal. \textbf{33} (1996), 1106--1124.

\bibitem[FFP]{FeiFuehPraet:2013}
M.~Feischl, T.~F\"uhrer, and D.~Praetorius, \emph{Adaptive {FEM} with optimal
  convergence rates for a certain class of non-symmetric and possibly
  non-linear problems}, SIAM J. Numer. Anal. \textbf{52} (2014), 601--625.

\bibitem[GR86]{GiraultRaviart:86}
V. Girault and P.-A. Raviart, \emph{Finite element methods for
  {N}avier-{S}tokes equations. {T}heory and algorithms}, {Springer Series in
  Computational Mathematics, 5}, 1986.

\bibitem[Gud10]{Gudi:2010}
T.~Gudi, \emph{A new error analysis for discontinuous finite element methods
  for linear elliptic problems}, Math. Comp. \textbf{79} (2010), no.~272,
  2169--2189.

\bibitem[HX13]{HuXu:2013}
J. Hu and J. Xu, \emph{Convergence and optimality of the adaptive
  nonconforming linear element method for the {S}tokes problem}, J. Sci.
  Comput. \textbf{55} (2013), no.~1, 125--148.

\bibitem[Kos94]{Kossaczky:94}
I.~Kossaczk\'y, \emph{A recursive approach to local mesh refinement in two and
  three dimensions}, J. Comput. Appl. Math. \textbf{55} (1994), 275--288.

\bibitem[KPP13]{KarkulikPavlicekPraetorius:13}
M. Karkulik, D. Pavlicek, and D. Praetorius, \emph{On 2{D} newest
  vertex bisection: optimality of mesh-closure and {$H^1$}-stability of
  {$L_2$}-projection}, Constr. Approx. \textbf{38} (2013), no.~2, 213--234.

\bibitem[KS11]{KreuzerSiebert:11}
C. Kreuzer and K.\,G. Siebert, \emph{Decay rates of adaptive finite
  elements with {D}\"orfler marking}, Numer. Math. \textbf{117} (2011),
  no.~4, 679--716.

\bibitem[Mau95]{Maubach:95}
J.\,M. Maubach, \emph{Local bisection refinement for n-simplicial grids
  generated by reflection}, SIAM J. Sci. Comput. \textbf{16} (1995), 210--227.

\bibitem[Mit89]{Mitchell:89}
W.\,F. Mitchell, \emph{A comparison of adaptive refinement techniques for
  elliptic problems}, ACM Trans. Math. Softw. \textbf{15} (1989), 326--347.

\bibitem[MN05]{MekchayNochetto:05}
K. Mekchay and R.\,H. Nochetto, \emph{Convergence of adaptive finite
  element methods for general second order linear elliptic {PDE}s}, SIAM J.
  Numer. Anal. \textbf{43} (2005), no.~5, 1803--1827.

\bibitem[MNS00]{MoNoSi:00}
P. Morin, R.\,H. Nochetto, and K.\,G. Siebert, \emph{Data
  oscillation and convergence of adaptive {FEM}}, SIAM J. Numer. Anal.
  \textbf{38} (2000), 466--488.

\bibitem[MNS02]{MoNoSi:02}
\bysame, \emph{Convergence of adaptive finite element methods}, SIAM Review
  \textbf{44} (2002), 631--658.

\bibitem[MNS03]{MoNoSi:03}
\bysame, \emph{Local problems on stars: {A} posteriori error estimators,
  convergence, and performance}, Math. Comp. \textbf{72} (2003),
1067--1097.

\bibitem[MSV08]{MoSiVe:08}
P. Morin, K.\,G. Siebert, and A. Veeser,
\emph{A basic convergence result for conforming adaptive finite
                  elements}, Math. Models Methods Appl. \textbf{18}
                (2008), no.~5, 707--737.

\bibitem[NSV09]{NoSiVe:09}
R.\,H. Nochetto, K.\,G. Siebert, and A. Veeser, \emph{Theory of
  adaptive finite element methods: An introduction}, Multiscale, Nonlinear and
  Adaptive Approximation (Ronald~A. DeVore and Angela Kunoth, eds.), Springer,
  2009, pp.~409--542.

\bibitem[Rab10]{Rabus:2010}
H.~Rabus, \emph{A natural adaptive nonconforming {FEM} of quasi-optimal
  complexity}, Comput. Methods Appl. Math. \textbf{10} (2010), no.~3, 315--325.

\bibitem[SS05]{SchmidtSiebert:05}
A. Schmidt and K.\,G. Siebert, \emph{Design of adaptive finite element
  software. {T}he finite element toolbox \textsf{ALBERTA}}, {Lecture Notes in
  Computational Science and Engineering 42. Springer}, 2005.

\bibitem[Sie11]{Siebert:11}
K.\,G. Siebert, \emph{A convergence proof for adaptive finite elements without lower
              bound}, IMA J. Numer. Anal. \textbf{31} (2011), no.~3,
            947--970.


\bibitem[Ste07]{Stevenson:07}
R. Stevenson, \emph{Optimality of a standard adaptive finite element method},
  Found. Comput. Math. \textbf{7} (2007), no.~2, 245--269.

\bibitem[Ste08]{Stevenson:08}
\bysame, \emph{The completion of locally refined simplicial partitions created
  by bisection}, Math. Comput. \textbf{77} (2008), no.~261, 227--241.

\bibitem[Tra97]{Traxler:97}
C.\,T. Traxler, \emph{An algorithm for adaptive mesh refinement in $n$
  dimensions}, Computing \textbf{59} (1997), 115--137.

\bibitem[Ver96]{Verfuerth:96}
R. Verf\"urth, \emph{A review of a posteriori error estimation and
  adaptive mesh-refinement techniques}, Adv. Numer. Math., John Wiley,
  Chichester, UK, 1996.

\bibitem[Ver13]{Verfuerth:13}
\bysame, \emph{A posteriori error estimation techniques for finite
  element methods}, Numerical Mathematics and Scientific Computation, The
  Clarendon Press Oxford University Press, New York, 2013.

\end{thebibliography}
\end{document}